\documentclass[11pt,oneside,reqno]{article}

\usepackage[margin=1.3in]{geometry}
\usepackage[title]{appendix}
\usepackage{amsmath,amsthm,amssymb,mathtools}
\usepackage{mathrsfs}
\usepackage{enumitem}
\usepackage[authoryear]{natbib}
\usepackage[colorlinks,citecolor=blue,urlcolor=blue,linkcolor=blue]{hyperref}

\hypersetup{
 pdftitle={A quantitative CLT for the zeros of random trigonometric polynomials},
 pdfauthor={Ao Huang, Zhonggen Su, Xiaolin Wang}
}

\numberwithin{equation}{section}
\allowdisplaybreaks[4]

\theoremstyle{plain}
\newtheorem{theorem}{Theorem}[section]

\newtheorem{proposition}{Proposition}[section]
\newtheorem{lemma}{Lemma}[section]

\theoremstyle{definition}

\makeatletter
\renewcommand\section{\@startsection{section}{1}{\z@}%
 {-3.5ex \@plus -1ex \@minus -.2ex}%
 {1.3ex \@plus .2ex}%
 {\centering\small\scshape\mathversion{bold}}}
\renewcommand\subsection{\@startsection{subsection}{2}{\z@}%
 {-3.5ex \@plus -1ex \@minus -.2ex}%
 {1ex \@plus .2ex}%
 {\normalfont\normalsize\bfseries\mathversion{bold}}}
\makeatother

\renewcommand{\le}{\leqslant}

\renewcommand{\ge}{\geqslant}

\newcommand{\E}{\mathbb E}
\newcommand{\Var}{\operatorname{Var}}
\newcommand{\Cov}{\operatorname{Cov}}

\title{A quantitative CLT for the zeros of random trigonometric polynomials}
\author{Ao Huang$^1$, Zhonggen Su$^1$, Xiaolin Wang$^2$}
\date{\it $^1$Zhejiang University, $^2$The Chinese University of Hong Kong \\[2ex]  }

\begin{document}

\maketitle
\begin{abstract}
Let \(X_N(t)\), \(t\in[0,\pi]\), be a stationary Gaussian
trigonometric polynomial, and let \(\mathcal Z_N\) denote the
number of zeros of the rescaled process \(X_N(t/N)\)
on \([0,N\pi]\). We establish the quantitative central limit theorem
\[
 d_{\mathrm W}\left(
 \frac{\mathcal Z_N-\mathbb E[\mathcal Z_N]}
 {\sqrt{\operatorname{Var}(\mathcal Z_N)}},Z\right)
 \le C\frac{(\log N)^2}{\sqrt N},
\]
where \(Z\) is standard Gaussian and \(d_{\mathrm W}\) denotes
the Wasserstein-1 distance.
The proof combines Ornstein--Uhlenbeck regularization of the
singular Kac representation with direct estimates of the covariance
between the Malliavin--Stein factor and Stein test functions.
A lattice lower bound of order \(N^{-1/2}\) shows that the
convergence rate is optimal up to logarithmic factors.

\smallskip

\end{abstract}

\section{Introduction and main results}

Let
\[
 X_N(t)=\frac1{\sqrt N}\sum_{n=1}^N
 \bigl(a_n\sin(nt)+b_n\cos(nt)\bigr),\qquad t\in[0,\pi],
\]
where \(a_n\) and \(b_n\) are independent standard Gaussian random
variables. We consider the microscopic rescaling \(Y_N(t):=X_N(t/N)\), \(t\in[0,T_N]\), where \(T_N:=N\pi\).
The process \(Y_N\) is centered and stationary, with covariance
\[
 r_N(\tau):=\mathbb E[Y_N(0)Y_N(\tau)]
 =\frac1N\sum_{n=1}^N\cos\left(\frac{n\tau}{N}\right).
\]
Moreover, for every fixed \(\tau\),
\[
 r_N(\tau)\longrightarrow r(\tau):=
 \int_0^1\cos(\lambda\tau)\,d\lambda=\frac{\sin\tau}{\tau},
\]
where \(r(0)=1\) by continuous extension.

Let $\mathcal Z_N=N^{Y_N}_{[0,T_N]}(0)$ denote the number of zeros of \(Y_N\) on \([0,T_N]\), and define
\[
 S_N=\frac{\mathcal Z_N-\mathbb E[\mathcal Z_N]}{\sqrt{T_N}},
 \qquad v_N^2=\operatorname{Var}(S_N).
\]
The asymptotic distribution of zero and level-crossing counts for random
trigonometric polynomials has been studied by several authors. \citet{GranvilleWigman2011} proved a central limit theorem for the
zero count of the stationary ensemble. \citet{AzaisLeon2013} extended this result to crossings
of arbitrary fixed levels by combining a Wiener chaos expansion with a
comparison to the stationary sinc process. \citet{AzaisDalmaoLeon2016} later established a related central
limit theorem for classical random
cosine polynomials, obtaining asymptotic Gaussianity through multiple
Wiener--It\^o integrals and the Peccati--Tudor theorem.

For the stationary model considered here, \citet{GranvilleWigman2011} and \citet{AzaisLeon2013} show that
\begin{equation}
 \label{eq:known-variance-asymptotics}
 v_N^2\longrightarrow V^2\in(0,\infty),
\end{equation}
and
\[
 \frac{\mathcal Z_N-\mathbb E[\mathcal Z_N]}
 {\sqrt{\operatorname{Var}(\mathcal Z_N)}}
 \xrightarrow{\mathrm d} Z,\qquad Z\sim\mathcal N(0,1).
\]
The purpose of the present paper is to make this convergence quantitative.

More recently, \citet{Gass2023} obtained cumulant asymptotics for
the zero-counting measures of a class of real Gaussian processes that
includes Gaussian trigonometric polynomials and the stationary sinc
process. In particular, these results describe the asymptotics of
fixed-order moments. \citet{AnconaGassLetendreStecconi2025} developed a broader
cumulant approach to the volumes of zero sets and to critical-point
counts of Gaussian fields. Estimates for each fixed moment do not,
without further control of their dependence on the order, yield an
error bound uniform over all Lipschitz test functions. 

\citet{CoutinPeralta2023} studied the quantitative convergence of the
zero count of a random trigonometric polynomial on a fixed microscopic
interval to that of the limiting sinc process. Here the microscopic
interval grows with \(N\), and the limiting law after centering and
normalization is Gaussian.

To quantify the rate of normal approximation, we introduce the
Wasserstein-1 distance. For integrable real-valued random variables \(F\) and \(G\), the
Wasserstein-1 distance is defined by
\[
 d_{\mathrm W}(F,G)=\sup_{\operatorname{Lip}(h)\le1}
 \left|\mathbb E[h(F)]-\mathbb E[h(G)]\right|.
\]
Our main result is the following.

\begin{theorem}
\label{thm:main}
There exist constants \(0<c\le C<\infty\) and \(N_0\ge2\)
such that, for every \(N\ge N_0\),
\begin{equation}
 \label{eq:main-bound}
 \frac{c}{\sqrt N}
 \le d_{\mathrm W}\left(\frac{S_N}{v_N},Z\right)
 \le C\frac{(\log N)^2}{\sqrt N},
\end{equation}
where \(Z\sim\mathcal N(0,1)\).
\end{theorem}

We do not determine whether the logarithmic factors in the upper bound of
\eqref{eq:main-bound} are necessary.

We now outline the proof strategy for the upper bound,
highlighting the main difficulties and how they are addressed. Writing \(\lambda_N=\operatorname{Var}(Y_N'(0))\),
we start
from the Kac representation
\[
 \mathcal Z_N
 =
 \sqrt{\lambda_N}\int_0^{T_N}
 \delta_0(Y_N(t))
 \left|
 \frac{Y_N'(t)}{\sqrt{\lambda_N}}
 \right|\,dt.
\]
This representation involves the singular local functional
\((x,y)\mapsto\delta_0(x)|y|\), which is distribution-valued
in the first coordinate and does not belong to \(L^2(\gamma_2)\).
Here \(\gamma\) denotes the standard Gaussian probability measure on
\(\mathbb R\), and \(\gamma_2=\gamma\otimes\gamma\). This singularity complicates the direct use of the Kac
representation for quantitative normal approximation.

A standard way to give rigorous meaning to such Kac representations is
to mollify the Dirac factor by a Gaussian approximate identity. This
approach is also used in the construction of the Wiener chaos
expansion of crossing counts, see
\citet{AzaisLeon2013,AzaisDalmaoLeon2016,AzaisWschebor2009}.

The qualitative arguments of \citet{AzaisLeon2013} and
\citet{AzaisDalmaoLeon2016} use finite truncations of the
Wiener chaos expansion, with the truncation level tending
to infinity only after taking the limit in \(N\).
To obtain a quantitative rate through this approach, one
would need to control both the truncation error and the
normal approximation error as the truncation level grows
with \(N\).

We therefore use the Ornstein--Uhlenbeck semigroup as a soft spectral
regularization. Let \(P_\rho\) denote the Ornstein--Uhlenbeck semigroup
on the underlying Gaussian space, using the same notation for its
finite-dimensional versions. We write \(P_\rho^{(2)}\) for the semigroup
acting independently on two coordinates. If \(S_N=\sum_{q\ge2} I_{q,N}\)
denotes the Wiener chaos expansion of the centered zero count, then
\[
 P_\rho S_N
 =
 \sum_{q\ge2}\rho^q I_{q,N},
 \qquad 0<\rho<1.
\]
Thus every chaos is retained, while higher orders are increasingly
damped. We set \(S_{N,\rho}=P_\rho S_N\).
At the level of the local Kac functional, this is equivalent to
\[
 F_\rho
 :=
 P_\rho^{(2)}
 \left(
 \delta_0\otimes|\cdot|-\frac1\pi
 \right).
\]
Consequently,
\[
 S_{N,\rho}
 =
 \frac{\sqrt{\lambda_N}}{\sqrt{T_N}}
 \int_0^{T_N}
 F_\rho\left(
 Y_N(t),
 \frac{Y_N'(t)}{\sqrt{\lambda_N}}
 \right)\,dt.
\]
For each fixed \(N\),
\(S_{N,\rho}\to S_N\) in \(L^2\) as \(\rho\uparrow1\).

By Mehler's formula,
\[
 P_\rho F(\xi)
 =
 \mathbb E_{\xi'}
 \left[
 F\left(
 \rho\xi+\sqrt{1-\rho^2}\,\xi'
 \right)
 \,\middle|\,\xi
 \right],
\]
where \(\xi'\) is an independent copy of the underlying Gaussian
coefficient vector. Thus the Ornstein--Uhlenbeck regularization may be
viewed as adding a small Gaussian perturbation to the entire
trigonometric polynomial and averaging over this perturbation.
At the level of the Kac density, the value coordinate undergoes the
familiar Gaussian mollification, while the derivative factor is
regularized at the same time. The precise kernels are introduced in
Section~\ref{sec:preliminaries}.

The coupling
\(\rho\xi+\sqrt{1-\rho^2}\,\xi'\) in Mehler's formula is closely related to the Gaussian interpolation
used by 
\cite{BeliaevMuirheadRivera2020}. They couple a Gaussian field with an
independent copy to study covariances of topological events. Their
interpolation motivates our use of the Mehler coupling as a
probabilistic tool. Here we apply it to regularize the singular
zero-count functional; its role in comparing zero counts will become
clear after we identify the Gaussian variables involved in the
Malliavin--Stein estimate.

For the normal approximation, we use the Malliavin--Stein method
and estimate the covariance with the Stein test function directly. Write
\[
 \sigma_{N,\rho}^2=\operatorname{Var}(S_{N,\rho}),\qquad
 \Gamma_{N,\rho}=\left\langle DS_{N,\rho},
 -DL^{-1}S_{N,\rho}\right\rangle_{\mathfrak H}.
\]
If \(f_h\) solves the normal Stein equation associated with a
\(1\)-Lipschitz function \(h\), then
\[
 \mathbb E\left[h\left(S_{N,\rho}/\sigma_{N,\rho}\right)\right]
       -\mathbb E[h(Z)]
 =-\sigma_{N,\rho}^{-2}
 \operatorname{Cov}\left(\Gamma_{N,\rho},
 f_h'\left(S_{N,\rho}/\sigma_{N,\rho}\right)\right).
\]
This follows from Malliavin integration by parts; see
\citet{NourdinPeccati2009,NourdinPeccati2012}. A variance bound for
\(\Gamma_{N,\rho}\) would square the increasingly concentrated
regularized Kac density. We instead expose a local term in
\(\Gamma_{N,\rho}\) and estimate its covariance with the test function
using first absolute moments.

To estimate this covariance, we first rewrite
\(\Gamma_{N,\rho}\) by integrating by parts in the time
variable. This yields an endpoint term and a double integral
of local kernels \(K_{N,\rho}(a,b)\); the precise representation
is given in Section~\ref{sec:local-kernel}. For each fixed pair \(a,b\), the kernel can be written as a
function of at most five independent Gaussian coordinates.
Resampling these coordinates while leaving the remaining
independent coordinates unchanged gives a copy of
\(S_{N,\rho}\) that is independent of the original kernel.
The covariance estimate therefore reduces to controlling
the change in \(S_{N,\rho}\) under this replacement.

Fixing these coordinates and keeping their orthogonal complement
unchanged gives two coupled fields \(Y_N^x\) and \(Y_N^y\), with zero
counts \(\mathcal Z_N^x\) and \(\mathcal Z_N^y\). The precise construction is given
in Section~\ref{subsec:conditional-fields}, in particular
\eqref{eq:conditional-process}. The pivotal-point
viewpoint of \cite{BeliaevMuirheadRivera2020} suggests how to compare
these counts. We join the conditional fields by the linear path
\(f_\theta(t):=(1-\theta)Y_N^x(t)+\theta Y_N^y(t),\;0\le\theta\le1.
\)
This path is distinct from the Mehler coupling used above. An
interior simple zero persists as \(\theta\) varies, so the number of
interior zeros can change only at a zero-level critical point,
\(f_\theta(t)=f_\theta'(t)=0.\)
Away from the conditioning times \(a,b\), a two-dimensional
Kac--Rice formula estimates the number of pairs
\((t,\theta)\) satisfying these equations, with endpoint
crossings treated separately. Near \(a\) and \(b\), the conditional covariance matrix of
the value and its first two derivatives need not remain
uniformly nondegenerate. We therefore bound the two local
zero counts separately using analyticity and Jensen's formula. Combining these estimates gives
\begin{equation}
 \label{eq:intro-replacement}
 \mathbb E\left|\mathcal Z_N^x-\mathcal Z_N^y\right|
 \le C(1+|x|+|y|)^2\log T_N.
\end{equation}

The logarithmic factor in this bound comes from the decay
estimates used to control the perturbation over the growing
interval. For each fixed \(j\), the covariance bounds have
the form
\[
 |r_N^{(j)}(\tau)|\le C_j\chi(\tau),
 \qquad |\tau|\le T_N,
\]
where \(\chi(0):=1\) and
\(\chi(\tau):=\min\{1,|\tau|^{-1}\}\) for \(\tau\ne0\).
The perturbation inherits this decay away from \(a\) and \(b\),
and integrating the resulting Rice estimate contributes
a factor of \(\log T_N\).

The Mehler coupling transfers \eqref{eq:intro-replacement} to the
regularized count. Resampling just the coordinates supporting \(K_{N,\rho}(a,b)\) then bounds its covariance with a Lipschitz
function of \(S_{N,\rho}\). A uniform weighted first moment of the
kernel supplies the other factor. Finally, integration over \(a,b\)
costs a second logarithm through the covariance envelope \(\chi(b-a)\).
This yields
\[
 d_{\mathrm W}\left(S_{N,\rho}/\sigma_{N,\rho},Z\right)
 \le C(c_0)\frac{(\log N)^2}{\sqrt N}
\]
whenever \(\sigma_{N,\rho}^2\ge c_0>0\), uniformly in
\(\rho\in[1/2,1)\). For each fixed \(N\), we can therefore let \(\rho\uparrow1\) without choosing a regularization scale depending
on \(N\).

The rest of the paper is organized as follows.
Section~\ref{sec:preliminaries} collects the properties of \(Y_N\)
needed below and introduces the Ornstein--Uhlenbeck regularization.
In Section~\ref{sec:local-kernel}, integration by parts reveals which
Gaussian variables enter the normal approximation argument.
Section~\ref{sec:replacement} then bounds the effect of changing
these variables on the zero count.
Section~\ref{sec:malliavin-stein} combines these estimates to prove
Theorem~\ref{thm:main}.
The supporting estimates are proved in
Appendices~\ref{app:projection-estimates} and
\ref{app:soft-kernel-details}.

Throughout the paper, \(c,C\) denote positive numerical constants
that may change from line to line. Parameter dependence is indicated
explicitly, as in \(C_j\), \(C_\rho\), and \(C(c_0)\). For \(p\ge1\)
and a random variable \(F\), we write \(\|F\|_p=\bigl(\E|F|^p\bigr)^{1/p}\).
For \(x\in\mathbb R^d\), \(|x|\) denotes the Euclidean norm. For a
real matrix \(A\), we write
\[
 \|A\|_{\mathrm{op}}=\sup_{|x|=1}|Ax|
 =\sqrt{\lambda_{\max}(A^\top A)}.
\]
The notation \(A\preceq B\) for symmetric matrices means that
\(B-A\) is positive semidefinite. We also use the abbreviation
\(L_T=\log T\).
\section{Preliminaries and quantitative reduction}
\label{sec:preliminaries}

In this section we introduce the notation and the Gaussian framework
used throughout the paper. After recording the covariance properties of
the rescaled process, we construct the Ornstein--Uhlenbeck
regularization of the Kac functional and identify its Wiener chaos
expansion. We then state the uniform normal approximation Proposition~\ref{prop:regularized-clt} to which
the proof of Theorem~\ref{thm:main} will be reduced.

\subsection{The Gaussian model and covariance estimates}

Recall that
\[
Y_N(t)
=
\frac1{\sqrt N}
\sum_{n=1}^N
\left(
a_n\sin\left(\frac{nt}{N}\right)
+
b_n\cos\left(\frac{nt}{N}\right)
\right),
\qquad
t\in[0,T_N],
\qquad
T_N=N\pi.
\]
The process \(Y_N\) is centered, stationary, and has covariance
\[
r_N(\tau)
=
\E[Y_N(0)Y_N(\tau)]
=
\frac1N
\sum_{n=1}^N
\cos\left(\frac{n\tau}{N}\right).
\]

Recall that \(\lambda_N=\operatorname{Var}(Y_N'(0))\).
By stationarity,
\[
\lambda_N
=
-r_N''(0)
=
\frac1N\sum_{n=1}^N\left(\frac nN\right)^2.
\]
We denote the standardized derivative by
\(V_N(t)=Y_N'(t)/\sqrt{\lambda_N}\).
Note that
\[
\Var(Y_N(t))=1,
\qquad
\Var(V_N(t))=1,
\qquad
\Cov(Y_N(t),V_N(t))=0.
\]
For every \(t\), the Gaussian vector
\(\boldsymbol{U}_N(t)
:=
\bigl(Y_N(t),V_N(t)\bigr)
\)
has law \(\gamma_2=\gamma\otimes\gamma\). 

For a compact interval \(I\subset\mathbb R\) and a level \(u\in\mathbb R\),
we denote by
\[
\mathcal Z_N(I;u)
=
N_I^{Y_N}(u)
=
\#\{t\in I:Y_N(t)=u\}
\]
the number of crossings of the level \(u\) by \(Y_N\) on \(I\), and write \(\mathcal Z_N(I) =\mathcal Z_N(I;0)\).

We recall the following form of Bulinskaya's lemma
\citep[Proposition~1.20]{AzaisWschebor2009}.
Let \(U\) be a real-valued process with \(C^1\) sample paths
on a compact interval \(J\). If the density \(p_{U(t)}(z)\)
is uniformly bounded for \(t\in J\) and \(z\) in a neighborhood
of zero, then almost surely there is no \(t\in J\) such that
\(U(t)=U'(t)=0\).

Applied to \(Y_N\), this gives the almost sure absence of multiple
zeros on compact intervals. The process also has no zero at a prescribed
deterministic time. Consequently, the choice of open, closed, or
half-open endpoints is immaterial for the unconditioned counts. A
nonzero trigonometric polynomial of degree \(N\) has at most \(2N\)
zeros over one period, counted with multiplicity. In particular,
\[
 0\le\mathcal Z_N\le2N\qquad\text{almost surely}.
\]
The zero count therefore belongs to \(L^p(\Omega)\) for every finite \(p\) at each fixed \(N\). 

The Rice formula
\citep[see, e.g.,][Chapter~3, Eq.~(3.2)]{AzaisWschebor2009} gives
\begin{equation}
\label{eq:mean-zero-count}
\E\mathcal Z_N(I)
=
\frac{\sqrt{\lambda_N}}{\pi}|I|
\end{equation}
for every compact interval \(I\). In particular,
\(\E\mathcal Z_N = N\sqrt{\lambda_N}\).

For \(k\ge1\), let
\[
m_{k,N}
=
\frac1N
\sum_{n=1}^N
\left(\frac nN\right)^k,
\]
and use the abbreviations
\[
\alpha_N=m_{2,N},
\qquad
\beta_N=m_{4,N},
\qquad
\gamma_N=m_{6,N},
\qquad
\delta_N=m_{8,N}.
\]
Thus \(\lambda_N=\alpha_N\), and
\[
\alpha_N\longrightarrow\frac13,
\qquad
\beta_N\longrightarrow\frac15,
\qquad
\gamma_N\longrightarrow\frac17,
\qquad
\delta_N\longrightarrow\frac19.
\]
Consequently, for all sufficiently large \(N\),
\begin{equation}
\label{eq:moment-uniform-bounds}
0<c\le
\alpha_N,\beta_N,\gamma_N,\delta_N
\le C,
\end{equation}
and
\begin{equation}
\label{eq:moment-determinants}
\beta_N-\alpha_N^2\ge c,
\qquad
\alpha_N\gamma_N-\beta_N^2\ge c.
\end{equation}

We shall repeatedly use the following two properties about the 
covariances.

\begin{lemma}
\label{lem:uniform-covariance-decay}
For every fixed integer \(j\ge0\), the following statements hold.

\begin{enumerate}
\item For every \(A<\infty\),
\[
\sup_{|\tau|\le A}
\left|
r_N^{(j)}(\tau)-r^{(j)}(\tau)
\right|
\longrightarrow0
\qquad\text{as }N\to\infty.
\]

\item There exists \(C_j<\infty\) such that, uniformly in \(N\) and
\(|\tau|\le T_N\),
\[
|r_N^{(j)}(\tau)|
\le
C_j\chi(\tau),
\]
where
\[
\chi(\tau)
=
\mathbf 1_{\{|\tau|\le1\}}
+
\frac1{|\tau|}
\mathbf 1_{\{|\tau|>1\}}.
\]
\end{enumerate}
\end{lemma}

\begin{proof}
The first assertion follows directly from the uniform convergence of the corresponding
Riemann sums on compact sets. Indeed,
\[
r_N^{(j)}(\tau)
=
\operatorname{Re}\left[
i^j\frac1N
\sum_{n=1}^N
\left(\frac nN\right)^j
e^{in\tau/N}
\right],
\]
whereas
\(r^{(j)}(\tau)
=
\operatorname{Re}\left[
i^j
\int_0^1
\lambda^j e^{i\lambda\tau}\,d\lambda
\right].
\)

For the second assertion, the estimate is immediate when \(|\tau|\le1\),
since
\[
|r_N^{(j)}(\tau)|
\le
\frac1N
\sum_{n=1}^N
\left(\frac nN\right)^j
\le1.
\]
Note that \(1<|\tau|\le T_N\)
and set \(x=\tau/N\). Let \(w_n=(n/N)^j\), then
Abel summation gives
\[
\left|
\sum_{n=1}^Nw_ne^{inx}
\right|
\le
\sup_{1\le k\le N}
\left|
\sum_{n=1}^ke^{inx}
\right|
\left(
|w_N|
+
\sum_{n=1}^{N-1}|w_{n+1}-w_n|
\right).
\]
Note, 
for \(|x|\le\pi\),
\[
\sup_{1\le k\le N}
\left|
\sum_{n=1}^ke^{inx}
\right|
\le
\frac{C_j}{|\sin(x/2)|}
\le
\frac{C_jN}{|\tau|}.
\]
And the second factor is at most 2.
After division by \(N\), this proves
\(|r_N^{(j)}(\tau)|\le C_j/|\tau|\).
\end{proof}

We next record two covariance matrices used below, together with bounds
that follow from the preceding covariance and spectral-moment estimates.

The cross-covariance matrix of \(\boldsymbol{U}_N\) is
\begin{equation}
\label{eq:normalized-covariance-matrix}
\mathcal R_N(\tau)
:=
\E\bigl[\boldsymbol{U}_N(0)\boldsymbol{U}_N(\tau)^\top\bigr]
=
\begin{pmatrix}
r_N(\tau)
&
r_N'(\tau)/\sqrt{\lambda_N}
\\[1mm]
-r_N'(\tau)/\sqrt{\lambda_N}
&
-r_N''(\tau)/\lambda_N
\end{pmatrix}.
\end{equation}
It follows from Lemma~\ref{lem:uniform-covariance-decay} and
\eqref{eq:moment-uniform-bounds} that
\begin{equation}
\label{Rmatrix_bound}
\|\mathcal R_N(\tau)\|_{\mathrm{op}}
\le
C\chi(\tau),
\qquad
|\tau|\le T_N.
\end{equation}

The covariance of the value and its first two derivatives at a
single time is
\begin{equation}
 \label{eq:three-jet-covariance}
 \operatorname{Cov}(Y_N(t),Y_N'(t),Y_N''(t))
 =\begin{pmatrix}
 1&0&-\lambda_N\\
 0&\lambda_N&0\\
 -\lambda_N&0&\beta_N
 \end{pmatrix}.
\end{equation}
Its eigenvalues belong to a fixed interval \([c,C]\subset(0,\infty)\)
for all sufficiently large \(N\), by
\eqref{eq:moment-uniform-bounds}--\eqref{eq:moment-determinants}.
The lower bound will also be used after conditioning on the values and
derivatives at two other times.

\subsection{Wiener chaos and Ornstein--Uhlenbeck regularization}

Let \(\mathfrak H = \ell^2(\mathbb N)\oplus\ell^2(\mathbb N)\),
and let \(W=\{W(h):h\in\mathfrak H\}\) be the isonormal Gaussian process
defined by \(W(e_n^{(1)})=a_n\) and \(W(e_n^{(2)})=b_n\),
where \(\{e_n^{(1)},e_n^{(2)};n\ge1\}\)
is the canonical orthonormal basis of \(\mathfrak H\).

Define
\[
e_{N,t}
=
\frac1{\sqrt N}
\sum_{n=1}^N
\left[
\sin\left(\frac{nt}{N}\right)e_n^{(1)}
+
\cos\left(\frac{nt}{N}\right)e_n^{(2)}
\right]
\]
and
\[
f_{N,t}
=
\frac1{\sqrt{N\lambda_N}}
\sum_{n=1}^N
\frac nN
\left[
\cos\left(\frac{nt}{N}\right)e_n^{(1)}
-
\sin\left(\frac{nt}{N}\right)e_n^{(2)}
\right].
\]
Then
\begin{equation*}
Y_N(t)=W(e_{N,t}),
\qquad
V_N(t)=W(f_{N,t}),
\end{equation*}
and
\[
\|e_{N,t}\|_{\mathfrak H}
=
\|f_{N,t}\|_{\mathfrak H}
=
1,
\qquad
\langle e_{N,t},f_{N,t}\rangle_{\mathfrak H}=0.
\]

All vectors used for a fixed \(N\) belong to the finite-dimensional
subspace
\[
 \mathfrak H_N:=\operatorname{span}
 \{e_n^{(1)},e_n^{(2)}:1\le n\le N\}\simeq\mathbb R^{2N}.
\]
Direct differentiation gives
\begin{equation}
 \label{eq:all-derivative-bounds}
 \|e_{N,t}^{(k)}\|_{\mathfrak H}^2
 =\frac1N\sum_{n=1}^N(n/N)^{2k}\le1,
 \qquad k\ge0,\quad t\in\mathbb R.
\end{equation}

Let \(H_q\) denote the Hermite polynomial, normalized by
\[
\E[H_p(Z)H_q(Z)]
=
q!\,\mathbf1_{\{p=q\}},
\qquad
Z\sim\mathcal N(0,1).
\]
Write $\varphi$ for the standard normal density and $\varphi_\eta$
for the density of $\mathcal N(0,\eta)$ when $\eta>0$.
The Hermite coefficients of \(|\cdot|\) are
\[
a_{2k}
:=
\frac1{(2k)!}
\E\bigl[|Z|H_{2k}(Z)\bigr],
\qquad
k\ge0,
\]
whereas all odd coefficients vanish. We also set
\[
b_k
=
\varphi(0)\frac{H_k(0)}{k!},
\qquad
k\ge0.
\]
Thus \(b_k=0\) for odd \(k\).

For \(q\ge0\), define
\begin{equation*}
f_q(x,y)
=
\sum_{k=0}^{\lfloor q/2\rfloor}
b_{q-2k}a_{2k}
H_{q-2k}(x)H_{2k}(y).
\end{equation*}
In the distributional Hermite sense,
\begin{equation}
\label{eq:Kac-Hermite-expansion}
\delta_0(x)|y|
=
\sum_{q\ge0}f_q(x,y).
\end{equation}
At level zero, \(f_q=0\) for every odd \(q\),
and
\(f_0(x,y)=b_0a_0=1/\pi\).

For a compact interval \(I\), let \(\mathcal F_N(I) = \mathcal Z_N(I)-\E\mathcal Z_N(I)\)
and, for \(q\ge1\), define
\begin{equation*}
J_{q,N}(I)
=
\sqrt{\lambda_N}
\int_I
f_q\bigl(\boldsymbol{U}_N(t)\bigr)\,dt.
\end{equation*}
The Hermite expansion for level crossings in
\citet[Proposition~1]{KratzLeon1997}
\citep[see also][Section~3, Eqs.~(3.3)--(3.7)]{AzaisLeon2013} gives
\begin{equation*}
\mathcal F_N(I)
=
\sum_{\substack{q\ge2\\q\ \mathrm{even}}}
J_{q,N}(I)
\qquad
\text{in }L^2(\Omega).
\end{equation*}
In particular,
\begin{equation*}
S_N
=
\sum_{\substack{q\ge2\\q\ \mathrm{even}}}
I_{q,N},
\qquad
I_{q,N}
=
\frac1{\sqrt{T_N}}
J_{q,N}([0,T_N]).
\end{equation*}

Next we introduce the regularized Kac density. For
\(0<\rho<1\), write
$$s_\rho=1-\rho^2$$
and define
\begin{equation*}
B_\rho(x)
=
P_\rho\delta_0(x)
=
\frac1{\sqrt{2\pi s_\rho}}
\exp\left(
-\frac{\rho^2x^2}{2s_\rho}
\right),
\end{equation*}
where the equality is understood in the distributional sense, and
\begin{equation}
\label{eq:A-rho-definition}
A_\rho(y)
=
P_\rho|\cdot|(y)
=
\E\left|
\rho y+\sqrt{s_\rho}Z
\right|,
\qquad
Z\sim\mathcal N(0,1).
\end{equation}
Set
\begin{equation*}
F_\rho(x,y)
=
B_\rho(x)A_\rho(y)-\frac1\pi.
\end{equation*}
Since the standard Gaussian measure is invariant under the
Ornstein--Uhlenbeck semigroup,
\[
\E[B_\rho(Z_1)A_\rho(Z_2)]
=
\varphi(0)\E|Z_2|
=
\frac1\pi
\]
for independent standard Gaussian variables \(Z_1,Z_2\). Hence
\(\E[F_\rho]=0\).

For a compact interval \(I\), define the uncentered regularized zero count
\begin{equation*}
\mathcal Z_{N,\rho}(I)
=
\sqrt{\lambda_N}
\int_I
B_\rho(Y_N(t))
A_\rho(V_N(t))
\,dt,
\end{equation*}
and its centered version
\begin{equation*}
\mathcal F_{N,\rho}(I)
=
\mathcal Z_{N,\rho}(I)
-
\E\mathcal Z_{N,\rho}(I)
=
\sqrt{\lambda_N}
\int_I
F_\rho(\boldsymbol{U}_N(t))\,dt.
\end{equation*}
Notice that
\begin{equation*}
\E\mathcal Z_{N,\rho}(I)
=
\frac{\sqrt{\lambda_N}}{\pi}|I|
=
\E\mathcal Z_N(I).
\end{equation*}

Define
\begin{equation}
\label{eq:S-N-rho-definition}
S_{N,\rho}
=
\frac1{\sqrt{T_N}}
\mathcal F_{N,\rho}([0,T_N]).
\end{equation}
The following lemma gives the Wiener chaos expansions of
\(\mathcal F_{N,\rho}(I)\) and \(S_{N,\rho}\).
\begin{lemma}[Chaos representation of the regularized count]
\label{lem:regularized-chaos-representation}
For every compact interval \(I\),
\begin{equation}
\label{eq:regularized-block-chaos}
\mathcal F_{N,\rho}(I)
=
\sum_{\substack{q\ge2\\q\ \mathrm{even}}}
\rho^qJ_{q,N}(I)
\qquad
\text{in }L^2(\Omega).
\end{equation}
In particular,
\begin{equation}
\label{eq:S-N-rho-chaos}
S_{N,\rho}
=
\sum_{\substack{q\ge2\\q\ \mathrm{even}}}
\rho^qI_{q,N}.
\end{equation}

Moreover, the function
\begin{equation}
\label{eq:G-rho-definition}
G_\rho
:=
\sum_{\substack{q\ge2\\q\ \mathrm{even}}}
\frac{\rho^q}{q}f_q
\end{equation}
belongs to \(L^2(\gamma_2)\) and satisfies
\begin{equation}
\label{eq:G-rho-integral}
G_\rho
=
\int_0^\rho\frac{F_t}{t}\,dt
\qquad
\text{in }L^2(\gamma_2).
\end{equation}
\end{lemma}

\begin{proof}
The Ornstein--Uhlenbeck semigroup multiplies the component of total Hermite
degree \(q\) by \(\rho^q\). Thus, by
\eqref{eq:Kac-Hermite-expansion},
\[
B_\rho(x)A_\rho(y)
=
\sum_{\substack{q\ge0\\q\ \mathrm{even}}}
\rho^qf_q(x,y)
\]
in \(L^2(\gamma_2)\). Subtracting the zeroth chaos yields
\[
F_\rho
=
\sum_{\substack{q\ge2\\q\ \mathrm{even}}}
\rho^qf_q.
\]

Let $R_Q
=
F_\rho
-
\sum_{\substack{2\le q\le Q\\q\ \mathrm{even}}}
\rho^qf_q.$ Since \(\boldsymbol{U}_N(t)\sim\gamma_2\), Minkowski's inequality gives
\[
\left\|
\sqrt{\lambda_N}
\int_I
R_Q(\boldsymbol{U}_N(t))\,dt
\right\|_2
\le
\sqrt{\lambda_N}|I|
\|R_Q\|_{L^2(\gamma_2)}
\longrightarrow0.
\]
This proves \eqref{eq:regularized-block-chaos} and
\eqref{eq:S-N-rho-chaos}.

The series in \eqref{eq:G-rho-definition} converges in \(L^2(\gamma_2)\),
since
\[
\sum_{\substack{q\ge2\\q\ \mathrm{even}}}
\frac{\rho^{2q}}{q^2}
\|f_q\|_{L^2(\gamma_2)}^2
\le
\|F_\rho\|_{L^2(\gamma_2)}^2.
\]
Moreover, ${\rho^q}/{q}
=
\int_0^\rho t^{q-1}\,dt,$ which proves \eqref{eq:G-rho-integral} by termwise integration in
\(L^2(\gamma_2)\).
\end{proof}

Since \(P_\rho\) multiplies the \(q\)-th Wiener chaos by
\(\rho^q\), the preceding lemma also shows that
\begin{equation}
 \label{eq:semigroup-on-count}
 S_{N,\rho}=P_\rho S_N.
\end{equation}
For each fixed \(N\), orthogonality gives
\begin{equation}
 \label{eq:fixed-N-OU-limit}
 \begin{aligned}
 \|S_N-S_{N,\rho}\|_2^2
 &=\sum_{\substack{q\ge2\\q\ \mathrm{even}}}
 (1-\rho^q)^2\|I_{q,N}\|_2^2
 \longrightarrow0
 \qquad\text{as }\rho\uparrow1.
 \end{aligned}
\end{equation}
Indeed, the summands are bounded by \(\|I_{q,N}\|_2^2\), whose sum
is \(v_N^2<\infty\). If \(\sigma_{N,\rho}^2=\Var(S_{N,\rho})\), the reverse triangle
inequality in \(L^2\) consequently yields
\begin{equation}
 \label{eq:variance-stability-fixed-N}
 |\sigma_{N,\rho}-v_N|
 \le\|S_{N,\rho}-S_N\|_2\longrightarrow0.
\end{equation}
Only this fixed-\(N\) convergence will be required when the
regularization is removed.

\subsection{The Malliavin--Stein identity and quantitative reduction}

We use \(D\), \(L\), and \(L^{-1}\) for the Malliavin derivative, the
Ornstein--Uhlenbeck generator, and its pseudo-inverse associated with the
isonormal process \(W\); see, for instance,
\citet{Nualart2006,NualartNualart2018,NourdinPeccati2012}
for background. We write \(\mathbb D^{1,2}\) for the Gaussian Sobolev
space of square-integrable random variables with square-integrable
\(\mathfrak H\)-valued Malliavin derivatives.

For a \(1\)-Lipschitz function \(h\), let \(f_h\) be the bounded solution of
\[
 f_h'(x)-xf_h(x)=h(x)-\E h(Z),\qquad Z\sim\mathcal N(0,1).
\]
The standard bounds for this solution are
\begin{equation}
 \label{eq:Stein-solution-bounds}
 \|f_h'\|_\infty\le1,\qquad \operatorname{Lip}(f_h')\le2;
\end{equation}
see~\citet[Lemma~1.2(v)]{NourdinPeccati2009}.
If \(F\in\mathbb D^{1,2}\) is centered, \(\sigma^2:=\Var(F)>0\), and \(\Gamma_F:=\langle DF,-DL^{-1}F\rangle_{\mathfrak H}\), then Malliavin
integration by parts gives
\begin{equation}
 \label{eq:Malliavin-IBP}
 \E\Gamma_F=\sigma^2,\qquad
 \E[(F/\sigma)f_h(F/\sigma)]
 =\sigma^{-2}\E[\Gamma_F f_h'(F/\sigma)].
\end{equation}
Consequently,
\begin{equation}
 \label{eq:direct-Stein-covariance}
 \E h(F/\sigma)-\E h(Z)
 =-\sigma^{-2}\Cov(\Gamma_F,f_h'(F/\sigma)).
\end{equation}
The identity is a direct consequence of
\citet[Eq.~(1.23)]{NourdinPeccati2009}. It can first be applied to smooth
Lipschitz test functions and then extended by approximation.

Applying the Cauchy--Schwarz inequality to
\eqref{eq:direct-Stein-covariance} gives
\[
 d_{\mathrm W}(F/\sigma,Z)
 \le\frac{\sqrt{\Var(\Gamma_F)}}{\sigma^2},
\]
whenever this variance is finite. Below we instead estimate the
covariance directly. We call \(\Gamma_F\) the Malliavin--Stein
factor. Its conditional expectation given \(F\) is a
one-dimensional Stein kernel.

The proof of Theorem~\ref{thm:main} is reduced to the following
proposition. We first identify the local kernel in
Section~\ref{sec:local-kernel}, control changes of the count after
replacing its Gaussian coordinates in Section~\ref{sec:replacement},
and combine these estimates in Section~\ref{sec:malliavin-stein}.

\begin{proposition}[Normal approximation of the regularized count]
\label{prop:regularized-clt}
For every \(c_0>0\), there exist constants \(C(c_0)<\infty\) and
\(N_0\ge1\) such that, for every \(N\ge N_0\) and
\(\rho\in[1/2,1)\) satisfying \(\sigma_{N,\rho}^2=\Var(S_{N,\rho})\ge c_0\),
one has
\begin{equation}
 \label{eq:regularized-CLT-bound}
 d_{\mathrm W}\left(S_{N,\rho}/\sigma_{N,\rho},Z\right)
 \le C(c_0)\frac{(\log N)^2}{\sqrt N}.
\end{equation}
\end{proposition}

To explain the reduction, let \(v_*>0\) be such that
\(v_N\ge v_*\) for all sufficiently large \(N\), as follows from
\eqref{eq:known-variance-asymptotics}. For each such fixed \(N\),
\eqref{eq:variance-stability-fixed-N} implies \(\sigma_{N,\rho}\ge v_*/2\) when \(\rho\) is sufficiently close
to one. The proposition can then be applied with \(c_0=v_*^2/4\). Its constant does not depend on how close \(\rho\)
must be to one. Consequently, the fixed-\(N\) limit
\eqref{eq:fixed-N-OU-limit} suffices to recover the exact zero count.
The normalization and the passage to the limit are carried out in
Subsection~\ref{subsec:proof-main}.

\section{The local Malliavin--Stein kernel}
\label{sec:local-kernel}

The identity \eqref{eq:direct-Stein-covariance} asks for the
covariance of \(\Gamma_{N,\rho}\) with a function of the regularized
count. We first write the Stein factor as a two-time integral. A
direct bound would involve the derivative of \(B_\rho\), which becomes
concentrated as \(\rho\uparrow1\). Time integration by parts removes this derivative from the
integral term and yields a local kernel containing a product
of two \(B_\rho\) factors.
The variables in that kernel will determine the replacement space in
Section~\ref{sec:replacement}.

\subsection{Integral representation of the Malliavin--Stein factor}\label{subsec:kernel-representation}

By Lemma~\ref{lem:regularized-chaos-representation} and the fact
that \(-L^{-1}\) acts on the \(q\)-th Wiener chaos by multiplication
by \(1/q\), we have
\begin{equation}
\label{eq:L-inverse-representation}
-L^{-1}S_{N,\rho}
=
\frac{\sqrt{\lambda_N}}{\sqrt{T_N}}
\int_0^{T_N}
G_\rho(\boldsymbol{U}_N(t))\,dt.
\end{equation}
Define
\[
 \Gamma_{N,\rho}=\left\langle DS_{N,\rho},
 -DL^{-1}S_{N,\rho}\right\rangle_{\mathfrak H}.
\]
For fixed \(\rho\in[1/2,1)\) and \(H=F_\rho,G_\rho\),
we have \(H(\boldsymbol{U}_N(t))\in\mathbb D^{1,2}\);
the required regularity is verified in
Appendix~\ref{app:soft-kernel-details}.
The Malliavin chain rule gives
\[
D[H(\boldsymbol{U}_N(t))]
=
\partial_1H(\boldsymbol{U}_N(t))e_{N,t}
+
\partial_2H(\boldsymbol{U}_N(t))f_{N,t}.
\]
Using \eqref{eq:L-inverse-representation} and
\eqref{eq:normalized-covariance-matrix}, we obtain
\begin{equation}
 \label{eq:gamma-integral-representation}
 \Gamma_{N,\rho}=\frac{\lambda_N}{T_N}
 \int_{[0,T_N]^2}\nabla F_\rho(\boldsymbol{U}_N(a))^\top
 \mathcal R_N(b-a)\nabla G_\rho(\boldsymbol{U}_N(b))\,da\,db.
\end{equation}
For each fixed \(\rho\in[1/2,1)\), the functions
\(F_\rho\) and \(G_\rho\) are continuously differentiable
with polynomially growing gradients; see
Appendix~\ref{app:soft-kernel-details}.
Since \(\boldsymbol{U}_N(t)\) depends linearly on the finite Gaussian
coefficient vector, these gradients evaluated at \(\boldsymbol{U}_N(t)\)
have moments of every order, uniformly in \(t\in[0,T_N]\),
for fixed \(N\) and \(\rho<1\).
These bounds justify differentiation under the time
integrals, Bochner integration, and Fubini's theorem
in \(L^2\). In particular, by \eqref{eq:Malliavin-IBP}, we have
\begin{equation*}
 \E\Gamma_{N,\rho}=\Var(S_{N,\rho})=\sigma_{N,\rho}^2.
\end{equation*}

\subsection{Time integration by parts}

A direct estimate of \eqref{eq:gamma-integral-representation} would
involve \(B_\rho'\). We first remove this derivative by integrating
in time.

Define \(H_N(t)=Y_N''(t)+\lambda_NY_N(t)\).
At each fixed time, \(H_N(t)\) is independent of \(\boldsymbol{U}_N(t)\), since
both cross-covariances vanish, and \(\Var(H_N(t))=\beta_N-\lambda_N^2\asymp1\).
Set
\begin{align}
 Q_{N,\rho}(a)&=B_\rho(Y_N(a))B_\rho(V_N(a))H_N(a),
 \label{eq:Q-kernel-definition}\\*
 J_{N,\rho}(a,b)&=r_N(b-a)\partial_1G_\rho(\boldsymbol{U}_N(b))
 +\frac{r_N'(b-a)}{\sqrt{\lambda_N}}\partial_2G_\rho(\boldsymbol{U}_N(b)),
 \notag\\*
 K_{N,\rho}(a,b)&=Q_{N,\rho}(a)J_{N,\rho}(a,b).
 \label{eq:K-kernel-definition}
\end{align}

\begin{lemma}[Time integration by parts]
For every \(N\) and \(\rho\in[1/2,1)\),
\begin{equation}
 \label{eq:DS-time-IBP}
 \begin{aligned}
 DS_{N,\rho}
 &=\frac1{\sqrt{T_N}}
 \left[B_\rho(Y_N(t))A_\rho'(V_N(t))e_{N,t}\right]_{t=0}^{t=T_N}\\
 &\quad-\frac{2\rho^2}{\sqrt{\lambda_NT_N}}
 \int_0^{T_N}Q_{N,\rho}(a)e_{N,a}\,da.
 \end{aligned}
\end{equation}
Consequently,
\begin{equation}
 \label{eq:gamma-IBP-representation}
 \Gamma_{N,\rho}=\mathcal B_{N,\rho}
 -\frac{2\rho^2}{T_N}
 \int_{[0,T_N]^2}K_{N,\rho}(a,b)\,da\,db,
\end{equation}
where
\begin{equation}
 \label{eq:boundary-definition}
 \mathcal B_{N,\rho}:=\frac{\sqrt{\lambda_N}}{T_N}
 \left[B_\rho(Y_N(a))A_\rho'(V_N(a))
 \int_0^{T_N}J_{N,\rho}(a,b)\,db\right]_{a=0}^{a=T_N}.
\end{equation}
The identity for \(DS_{N,\rho}\) holds in
\(L^2(\Omega;\mathfrak H)\), and that for
\(\Gamma_{N,\rho}\) holds in \(L^2(\Omega)\).
\end{lemma}

\begin{proof}
Splitting the Gaussian integral in \eqref{eq:A-rho-definition} gives
\[
 A_\rho(y)=\rho y\{2\Phi(\rho y/\sqrt{s_\rho})-1\}+2s_\rho B_\rho(y),
\]
where \(\Phi\) is the standard normal distribution function.
Differentiating, we obtain
\begin{equation}
 \label{eq:A-exact-identities}
 A_\rho''(y)=2\rho^2B_\rho(y),\qquad
 A_\rho(y)-yA_\rho'(y)=2s_\rho B_\rho(y).
\end{equation}
Also, \(B_\rho'(x)=-(\rho^2x/s_\rho)B_\rho(x)\).
Since \(e_{N,t}'=\sqrt{\lambda_N} f_{N,t}\) and \(V_N'=Y_N''/\sqrt{\lambda_N}\),
\begin{align*}
 &\frac d{dt}\{B_\rho(Y_N(t))A_\rho'(V_N(t))e_{N,t}\}\\*
 =&\sqrt{\lambda_N} B_\rho'(Y_N(t))V_N(t)A_\rho'(V_N(t))e_{N,t}+\lambda_N^{-1/2}B_\rho(Y_N(t))A_\rho''(V_N(t))Y_N''(t)e_{N,t}\\*
 &+\sqrt{\lambda_N} B_\rho(Y_N(t))A_\rho'(V_N(t))f_{N,t}.
\end{align*}
On the other hand,
\[
 DS_{N,\rho}=\frac{\sqrt{\lambda_N}}{\sqrt{T_N}}\int_0^{T_N}
 \{B_\rho'(Y_N(t))A_\rho(V_N(t))e_{N,t}
 +B_\rho(Y_N(t))A_\rho'(V_N(t))f_{N,t}\}\,dt.
\]
Subtracting the time derivative inside this integral, the remaining
coefficient of \(e_{N,t}\), before division by \(\sqrt{T_N}\), is
\begin{align*}
 &\sqrt{\lambda_N} B_\rho'(Y_N(t))
 \{A_\rho(V_N(t))-V_N(t)A_\rho'(V_N(t))\}-\lambda_N^{-1/2}B_\rho(Y_N(t))A_\rho''(V_N(t))Y_N''(t)\\*
 &\quad=-\frac{2\rho^2}{\sqrt{\lambda_N}}
 B_\rho(Y_N(t))B_\rho(V_N(t))\{Y_N''(t)+\lambda_N Y_N(t)\},
\end{align*}
by \eqref{eq:A-exact-identities}. This proves \eqref{eq:DS-time-IBP}.

For fixed \(N\) and \(\rho<1\), the smoothness of the
sample paths and the regularized kernels justifies
the integration by parts pathwise.
The polynomial growth bounds ensure that the resulting
identity also holds in \(L^2(\Omega;\mathfrak H)\). Differentiating \eqref{eq:L-inverse-representation} and using the
Gaussian Sobolev chain rule, we obtain
\begin{align*}
-DL^{-1}S_{N,\rho}
&=
\frac{\sqrt{\lambda_N}}{\sqrt{T_N}}
\int_0^{T_N}
\Bigl\{
\partial_1G_\rho(\boldsymbol{U}_N(b))e_{N,b}
+
\partial_2G_\rho(\boldsymbol{U}_N(b))f_{N,b}
\Bigr\}\,db.
\end{align*}
Moreover,
\begin{align*}
\langle e_{N,a},e_{N,b}\rangle_{\mathfrak H}
&=r_N(b-a),\qquad
\langle e_{N,a},f_{N,b}\rangle_{\mathfrak H}
=\frac{r_N'(b-a)}{\sqrt{\lambda_N}},
\end{align*}
and hence
\begin{align*}
&\left\langle
e_{N,a},
\partial_1G_\rho(\boldsymbol{U}_N(b))e_{N,b}
+
\partial_2G_\rho(\boldsymbol{U}_N(b))f_{N,b}
\right\rangle_{\mathfrak H}\\
\qquad=&
r_N(b-a)\partial_1G_\rho(\boldsymbol{U}_N(b))
+
\frac{r_N'(b-a)}{\sqrt{\lambda_N}}
\partial_2G_\rho(\boldsymbol{U}_N(b))=
J_{N,\rho}(a,b).
\end{align*}
Taking the inner product of \eqref{eq:DS-time-IBP} with
\(-DL^{-1}S_{N,\rho}\), and using Fubini's theorem, therefore gives
\begin{align*}
\Gamma_{N,\rho}
&=
\frac{\sqrt{\lambda_N}}{T_N}
\left[
B_\rho(Y_N(a))A_\rho'(V_N(a))
\int_0^{T_N}J_{N,\rho}(a,b)\,db
\right]_{a=0}^{a=T_N}\\
&\quad
-\frac{2\rho^2}{T_N}
\int_{[0,T_N]^2}
Q_{N,\rho}(a)J_{N,\rho}(a,b)\,da\,db.
\end{align*}
Using \(K_{N,\rho}=Q_{N,\rho}J_{N,\rho}\), this is precisely
\eqref{eq:gamma-IBP-representation}, with the boundary term given by
\eqref{eq:boundary-definition}. No cancellation of the two endpoint
terms is assumed.
\end{proof}

To estimate \(J_{N,\rho}(a,b)\) using covariance decay,
we use the following uniform gradient bound, whose proof
is given in Appendix~\ref{app:soft-kernel-details}.

\begin{lemma}
\label{lem:kernel-estimates}
There exists \(C<\infty\) such that, for every
\(\rho\in[1/2,1)\),
\begin{equation*}
 |\nabla G_\rho(x,y)|\le C(1+|y|),
 \qquad (x,y)\in\mathbb R^2.
\end{equation*}
\end{lemma}
Since \(\lambda_N\) is bounded away from zero,
\eqref{Rmatrix_bound}
and Lemma~\ref{lem:kernel-estimates} imply
\begin{equation}
 \label{eq:J-kernel-envelope}
 |J_{N,\rho}(a,b)|
 \le C\chi(b-a)(1+|V_N(b)|).
\end{equation}

\section{Finite-rank replacement for the local kernel}
\label{sec:replacement}

To control the integral term in
\eqref{eq:gamma-IBP-representation}, we need the bound
\[
 |\Cov(K_{N,\rho}(a,b),g(S_{N,\rho}))|
 \le C\operatorname{Lip}(g)\,\chi(b-a)
       \frac{L_{T_N}}{\sqrt{T_N}}
\]
for every Lipschitz function \(g\), uniformly in
\(\rho\in[1/2,1)\).

By \eqref{eq:Q-kernel-definition}--\eqref{eq:K-kernel-definition},
\(K_{N,\rho}(a,b)\) depends only on
\(Y_N(a),Y_N'(a),Y_N''(a),Y_N(b),Y_N'(b)\).
Let \(\boldsymbol{\xi}_N=(a_1,\ldots,a_N,b_1,\ldots,b_N)\in\mathbb R^{2N}\) denote the standard Gaussian coefficient vector. After an orthogonal change of coordinates, these variables are functions of at most five independent standard Gaussian coordinates. Replacing the corresponding coordinates in
\(S_{N,\rho}\) by an independent copy, while keeping the
orthogonal coordinates unchanged, yields a statistic
independent of \(K_{N,\rho}(a,b)\).
The covariance can therefore be controlled through the
change in \(S_{N,\rho}\) caused by this replacement.

We first compare the exact zero counts obtained from
different choices of these coordinates while keeping
all remaining coordinates unchanged. The resulting estimate is stated
in Proposition~\ref{prop:replacement} below.
We then use Mehler's formula to transfer this estimate
to \(S_{N,\rho}\), uniformly in \(\rho\).
Combined with the weighted moment bound for
\(K_{N,\rho}(a,b)\) established in
Section~\ref{sec:malliavin-stein}, this yields the
covariance bound above.

\subsection{The coordinates selected by the kernel}
\label{subsec:conditional-fields}

For \(a,b\in[0,T_N]\), let
\begin{equation}
 \label{eq:two-site-subspace}
 \mathfrak H_{a,b}^{(N)}=\operatorname{span}
 \{e_{N,a},e_{N,a}',e_{N,a}'',e_{N,b},e_{N,b}'\}
 \subset\mathfrak H_N,
\end{equation}
and denote its orthogonal projection by \(\Pi_{a,b}^{(N)}\).
Under the isometry \(W\), this is the Gaussian subspace generated by \(Y_N(a),Y_N'(a),Y_N''(a),Y_N(b),Y_N'(b)\). Its dimension is at most
five; when \(a=b\), it is three for all sufficiently large \(N\).

Fix \(a,b\), and write \(p=\dim\mathfrak H_{a,b}^{(N)}\). Choose an
isometry \(\mathsf Q:\mathbb R^p\to\mathfrak H_{a,b}^{(N)}\).
In coefficient coordinates, decompose
\begin{equation}
 \label{eq:coefficient-decomposition}
 \boldsymbol{\xi}_N=\mathsf QX+\eta,
 \qquad X\sim\mathcal N(0,I_p),\qquad
 \eta\sim\mathcal N(0,I_{2N}-\Pi_{a,b}^{(N)}),
\end{equation}
where \(X\) and \(\eta\) are independent. Set
\begin{equation}
 \label{eq:conditional-process}
 \boldsymbol v(t)=\mathsf Q^*e_{N,t},\qquad
 R(t)=\langle e_{N,t},\eta\rangle,\qquad
 Y_N^x(t)=R(t)+\boldsymbol v(t)\cdot x.
\end{equation}
Here \(\mathsf Q^*\) denotes the adjoint. Throughout this section, we suppress the dependence of
\(p,\mathsf Q,R,\boldsymbol v\) on \(N,a,b\).
For a deterministic
\(x\in\mathbb R^p\), write
\[
 \mathcal Z_N^x(I)=\#\{t\in I:Y_N^x(t)=0\},\qquad
 \mathcal Z_N^x=\mathcal Z_N^x([0,T_N]).
\]
An arbitrary value may be
assigned on the event of an identically zero path, which will have
probability zero under the conditions used below.

We also regard \(\mathcal Z_N\), \(S_N\), and \(S_{N,\rho}\) as functions
of the coefficient vector, using the same notation for their evaluations.
For fixed \(a,b,\rho\), the definition of \(K_{N,\rho}(a,b)\)
implies that it can be written as \(k_{a,b,\rho}(X)\) for a measurable
function \(k_{a,b,\rho}\). If \(X'\) is an independent copy of \(X\),
also independent of \(\eta\), the vector \(\mathsf QX'+\eta\) has
the law of \(\boldsymbol{\xi}_N\) and is independent of \(X\). Hence, for a
Lipschitz function \(g\),
\begin{equation}
 \label{eq:local-resampling-identity}
 \begin{aligned}
 &\Cov(K_{N,\rho}(a,b),g(S_{N,\rho}))=\E\bigl[k_{a,b,\rho}(X)\bigl\{
 g(S_{N,\rho}(\mathsf QX+\eta))
 -g(S_{N,\rho}(\mathsf QX'+\eta))\bigr\}\bigr].
 \end{aligned}
\end{equation}
Thus the global covariance has been reduced to comparing the conditional fields \(Y_N^x\) and \(Y_N^y\), which share the same residual field \(R\). We first control the change in their exact zero counts and then use \eqref{eq:Mehler-coefficients} to transfer this estimate to the regularized statistic \(S_{N,\rho}\).

The perturbation from \(Y_N^x\) to \(Y_N^y\) is the deterministic function \(\boldsymbol v(t)\cdot(y-x)\).
To control its effect on the zero count over the full interval, we need the perturbation, together with its first two derivatives, to decay away from the conditioning times \(a\) and \(b\). This decay must remain uniform when \(b\) approaches \(a\), where the raw two-site coordinates become nearly linearly dependent. The required uniform control is provided by the following projection estimate.

\begin{lemma}
\label{lem:projection-tail}
There exist constants \(C<\infty\) and \(N_0\ge1\) such that, for
every \(N\ge N_0\), every \(a,b,t\in[0,T_N]\), and \(j=0,1,2\),
\begin{equation}
 \label{eq:projection-tail}
 \bigl\|\Pi_{a,b}^{(N)}e_{N,t}^{(j)}\bigr\|_{\mathfrak H}
 \le C\{\chi(t-a)+\chi(t-b)\}.
\end{equation}
\end{lemma}

The proof is given in Appendix~\ref{app:projection-estimates}.
The main difficulty is to obtain a bound uniform as \(b\to a\),
when the spanning vectors in \eqref{eq:two-site-subspace}
become nearly linearly dependent. For nearby distinct \(a,b\),
we resolve this by using divided differences whose Gram matrix
has a uniformly bounded inverse.

Since orthogonal projections are contractions,
\eqref{eq:all-derivative-bounds} implies
\begin{equation}
 \label{eq:conditional-all-derivatives}
 |\boldsymbol v^{(k)}(t)|\le1,\qquad
 \E|R^{(k)}(t)|^2\le1,
 \qquad k\ge0.
\end{equation}
Moreover, Lemma~\ref{lem:projection-tail} gives
\begin{equation}
 \label{eq:regression-tail}
 |\boldsymbol v^{(j)}(t)|
 \le C\{\chi(t-a)+\chi(t-b)\},\qquad j=0,1,2.
\end{equation}

We next show that the Gaussian vector
\((R(t),R'(t),R''(t))\) is uniformly nondegenerate when \(t\)
is sufficiently far from \(a\) and \(b\).
For \(A>1\), to be chosen below, set
\begin{equation*}
 \begin{split}
 d_{a,b}(t)&=\chi(t-a)+\chi(t-b),\\
 U_{a,b}&=\bigl([a-A,a+A]\cup[b-A,b+A]\bigr)\cap[0,T_N].
 \end{split}
\end{equation*}
The set \([0,T_N]\setminus U_{a,b}\) consists of at most three
intervals. The following lemma shows that \(A\) can be chosen
independently of \(N,a,b\) so that the required nondegeneracy
holds outside \(U_{a,b}\).

\begin{lemma}
\label{lem:far-residual}
There are constants \(A>1\), \(c,C>0\), and \(N_0\ge1\) such that,
for every \(N\ge N_0\), \(a,b\in[0,T_N]\), and
\(t\in[0,T_N]\) with
\(\min\{|t-a|,|t-b|\}\ge A\),
\begin{equation}
 \label{eq:far-residual-nondegeneracy}
 cI_3\preceq\operatorname{Cov}(R(t),R'(t),R''(t))\preceq CI_3.
\end{equation}
\end{lemma}

\begin{proof}
Let \(M_N\) be the covariance matrix in
\eqref{eq:three-jet-covariance}. Orthogonality in
\eqref{eq:coefficient-decomposition} gives
\[
 \operatorname{Cov}(R(t),R'(t),R''(t))
 =M_N-\bigl(\boldsymbol v^{(i)}(t)\cdot
                  \boldsymbol v^{(j)}(t)\bigr)_{0\le i,j\le2}.
\]
The matrix $G(t):=\bigl(\boldsymbol v^{(i)}(t)\cdot
\boldsymbol v^{(j)}(t)\bigr)_{0\le i,j\le2}$ is the Gram matrix of
\(\boldsymbol v(t),\boldsymbol v'(t),\boldsymbol v''(t)\), and is therefore
positive semidefinite. Hence
\[
\|G(t)\|_{\mathrm{op}}
\le \operatorname{Tr}G(t)
=
\sum_{j=0}^2|\boldsymbol v^{(j)}(t)|^2.
\]
By \eqref{eq:regression-tail}, if
\(\min\{|t-a|,|t-b|\}\ge A>1\), then $\sum_{j=0}^2|\boldsymbol v^{(j)}(t)|^2
\le C d_{a,b}(t)^2
\le CA^{-2}.$

By \eqref{eq:three-jet-covariance} and
\eqref{eq:moment-uniform-bounds}--\eqref{eq:moment-determinants},
there exists \(c_0>0\) such that $\lambda_{\min}(M_N)\ge c_0$ for all sufficiently large \(N\). Therefore,
\[
\lambda_{\min}\bigl(M_N-G(t)\bigr)
\ge \lambda_{\min}(M_N)-\|G(t)\|_{\mathrm{op}}
\ge c_0-CA^{-2}.
\]
Choosing \(A\) so that \(CA^{-2}\le c_0/2\) gives $M_N-G(t)\succeq {c_0}I_3/{2},$ which proves the lower bound in
\eqref{eq:far-residual-nondegeneracy}. Since \(G(t)\succeq0\) and
\(M_N-G(t)\) is a covariance matrix, we have $0\preceq M_N-G(t)\preceq M_N.$ The uniform upper bound in
\eqref{eq:far-residual-nondegeneracy} therefore follows from the
corresponding uniform bound on \(M_N\).
\end{proof}

\subsection{Stability of the zero count}

With \(A\) fixed as in Lemma~\ref{lem:far-residual}, we prove the
following estimate for the conditional fields in
\eqref{eq:conditional-process}, using the same \(\eta\)
for both choices of the conditioned coordinates.

\begin{proposition}
\label{prop:replacement}
There exist constants \(C<\infty\) and \(N_0\ge1\) such that, for
every \(N\ge N_0\), \(a,b\in[0,T_N]\), and
\(x,y\in\mathbb R^p\),
\begin{equation}
 \label{eq:replacement}
 \E_\eta|\mathcal Z_N^x-\mathcal Z_N^y|
 \le C(1+|x|+|y|)^2L_{T_N}.
\end{equation}
\end{proposition}

There are two regions to consider in proving this estimate.
Far from \(a,b\), the covariance bounds in
Lemma~\ref{lem:far-residual} allow us to apply a
two-dimensional Kac--Rice formula to the value and
derivative of the linear interpolation between
\(Y_N^x\) and \(Y_N^y\). This estimates the number of
interior pairs \((t,\theta)\) at which both vanish.
Together with endpoint crossings, these pairs control
the change in the zero count. Integrating the decay
bound \(d_{a,b}(t)\) over this region yields the factor
\(L_{T_N}\) in the replacement estimate.
On \(U_{a,b}\), we bound the two zero counts separately,
using analytic continuation and Jensen's formula. The next two lemmas provide the estimates for these two regions.

\begin{lemma}
\label{lem:far-count-change}
Let \(A\) and \(N_0\) be as in
Lemma~\ref{lem:far-residual}. There exists a constant
\(C<\infty\) such that, for every \(N\ge N_0\),
\(a,b\in[0,T_N]\), \(x,y\in\mathbb R^p\), and every
compact interval \(I\subset[0,T_N]\) satisfying
\(\min\{|t-a|,|t-b|\}\ge A\) for all \(t\in I\), we have
\begin{equation}
 \label{eq:far-count-change}
 \E_\eta|\mathcal Z_N^x(I)-\mathcal Z_N^y(I)|
 \le 2+C|x-y|(1+|x|+|y|)\int_I d_{a,b}(t)\,dt.
\end{equation}
\end{lemma}

\begin{proof}
If \(I\) is a singleton, both counts vanish almost surely
by Lemma~\ref{lem:far-residual}. We therefore assume that
\(I\) has positive length. Fix \(N,a,b,x,y\), and set
\[
f_\theta(t)
=
R(t)+\boldsymbol v(t)\cdot\bigl((1-\theta)x+\theta y\bigr),
\qquad
h(t)=\boldsymbol v(t)\cdot(y-x).
\]
Thus \(f_\theta=f_0+\theta h\), with \(f_0=Y_N^x\) and
\(f_1=Y_N^y\). All expectations below are with respect to
\(\eta\), and primes denote derivatives in \(t\) with
\(\theta\) fixed. If \(h\equiv0\), the counts agree and
the result is immediate.

We count the interior pairs \((t,\theta)\) at which
\(f_\theta(t)=f_\theta'(t)=0\). Write
\[
K=I\times[0,1],
\qquad
G(t,\theta)=
\begin{pmatrix}
f_\theta(t)\\
f_\theta'(t)
\end{pmatrix}.
\]
By Lemma~\ref{lem:far-residual},
\[
\operatorname{Cov}(G(t,\theta))
=
\operatorname{Cov}\bigl((R(t),R'(t))^\top\bigr)
\succeq cI_2,
\qquad (t,\theta)\in K.
\]
The trigonometric representation and continuity allow us
to extend \(G\), using the same formulas, to a bounded
open neighborhood of \(K\) on which its covariance is
bounded below by \((c/2)I_2\). Its sample paths are smooth,
and its value densities are uniformly bounded there.
The Bulinskaya-type lemma
\citep[Proposition~2.2]{ArmentanoAzaisLeon2025}
therefore implies that, almost surely, every zero of
\(G\) in \(K\) is regular. Since
\[
D_{(t,\theta)}G(t,\theta)
=
\begin{pmatrix}
f_\theta'(t)&h(t)\\
f_\theta''(t)&h'(t)
\end{pmatrix},
\]
at every such zero we have
\[
\det D_{(t,\theta)}G(t,\theta)
=
-h(t)f_\theta''(t)\ne0.
\]
Since \(G^{-1}(0)\cap K\) is compact and every one of its
points is isolated, it is almost surely finite.
In particular, no \(f_\theta\), \(\theta\in[0,1]\),
is identically zero on \(I\).
Let
\[
M_I
=
\#\bigl\{
(t,\theta)\in\operatorname{int}(I)\times(0,1):
G(t,\theta)=0
\bigr\}.
\]
The Gaussian Kac--Rice formula
\citep[Theorem~2.2]{ArmentanoAzaisLeon2025}
gives
\begin{equation*}
\begin{aligned}
\E M_I
={}&
\int_I\int_0^1
|h(t)|\,p_{f_\theta(t),f_\theta'(t)}(0,0)\\
&\qquad\times
\E\left[
|f_\theta''(t)|
\,\middle|\,
f_\theta(t)=f_\theta'(t)=0
\right]
\,d\theta\,dt.
\end{aligned}
\end{equation*}
The same formula applied to \(\partial K\), which has
zero two-dimensional Lebesgue measure, shows that
\(G\) has no zeros on \(\partial K\), almost surely.

By \eqref{eq:conditional-all-derivatives},
\[
|\E f_\theta^{(j)}(t)|\le |x|+|y|,
\qquad j=0,1,2.
\]
The uniform covariance bounds in
Lemma~\ref{lem:far-residual} imply
\[
p_{f_\theta(t),f_\theta'(t)}(0,0)\le C.
\]
Gaussian regression also shows that the conditional mean
of \(f_\theta''(t)\), given
\(f_\theta(t)=f_\theta'(t)=0\), has absolute value at most
\(C(|x|+|y|)\), while its conditional variance is at most
\(\Var(R''(t))\le1\). Hence
\[
\E\left[
|f_\theta''(t)|
\,\middle|\,
f_\theta(t)=f_\theta'(t)=0
\right]
\le C(1+|x|+|y|).
\]
Together with
\[
|h(t)|\le C|x-y|d_{a,b}(t),
\]
which follows from \eqref{eq:regression-tail}, these bounds
yield
\[
\E M_I
\le
C|x-y|(1+|x|+|y|)
\int_I d_{a,b}(t)\,dt
<\infty.
\]

It remains to relate \(M_I\) to the change in the zero
count. At a pair \((t_*,\theta_*)\) counted by \(M_I\),
we have \(h(t_*)\ne0\) and \(f_{\theta_*}''(t_*)\ne0\).
The implicit function theorem writes the local zero set
of \(f_\theta(t)\) as \(\theta=\psi(t)\), where
\[
\psi'(t_*)=0,
\qquad
\psi''(t_*)
=
-\frac{f_{\theta_*}''(t_*)}{h(t_*)}\ne0.
\]
Thus two simple zeros are created or removed as
\(\theta\) crosses \(\theta_*\). Each such pair contributes
at most two to the change in the zero count; the double
zero at the critical parameter itself is counted once.

At each endpoint \(u\) of \(I\), the variables \(f_0(u)\)
and \(f_1(u)\) are nondegenerate Gaussian and hence are
almost surely nonzero. The affine function
\[
\theta\longmapsto f_\theta(u)=f_0(u)+\theta h(u)
\]
therefore vanishes at most once. Any resulting zero is
simple in \(t\), since \(G\) has no zeros on \(\partial K\).
Consequently, each endpoint contributes at most one to
the change in the count. The absence of zeros of \(G\)
on \(I\times\{0,1\}\) also ensures that \(f_0\) and \(f_1\)
have no multiple zeros on \(I\).

Away from these finitely many parameter values, the
implicit function theorem and compactness of \(I\)
imply that the zero count is locally constant.
Summing the local contributions, including those that
occur at the same parameter value, gives
\begin{equation*}
\bigl|\mathcal Z_N^x(I)-\mathcal Z_N^y(I)\bigr|
\le 2+2M_I
\qquad\text{almost surely}.
\end{equation*}
Taking expectations and using the bound for \(\E M_I\)
proves \eqref{eq:far-count-change}.
\end{proof}

We next turn to \(U_{a,b}\), where uniform nondegeneracy
may fail. The following lemma bounds each zero count
separately using analyticity and Jensen's formula.

\begin{lemma}
\label{lem:near-count-bound}
There exists \(C<\infty\) such that, for every
\(N\ge N_0\), \(a,b\in[0,T_N]\), and \(x\in\mathbb R^p\),
\begin{equation}
 \label{eq:local-Jensen-bound}
 \E_\eta\mathcal Z_N^x(U_{a,b})
 \le C\{1+\log(1+|x|)\}.
\end{equation}
\end{lemma}

\begin{proof}
By increasing \(N_0\) if necessary, we may assume that
\(T_N\ge24A\). This will allow us to choose a center \(c\)
for Jensen's formula such that
\(\operatorname{Var}(R(c))\ge c_0>0\).
First consider
\(I_a=[a-A,a+A]\cap[0,T_N]\). We choose a deterministic point
\(c\in[0,T_N]\) such that \(|c-a|\le12A\) and
\(\min\{|c-a|,|c-b|\}\ge A\). To see that this is possible, suppose
first that \(a\le T_N/2\). Since \(T_N\ge24A\), the three points
\(a+4A,a+8A,a+12A\) belong to \([0,T_N]\). They are separated by
\(4A\), so at most one is within distance \(A\) of \(b\), and any
of the others is admissible. If \(a>T_N/2\), use the corresponding
negative shifts instead. Lemma~\ref{lem:far-residual} then gives
\begin{equation*}
\Var(R(c))\ge c_0>0.
\end{equation*}

Set \(r=16A\). For every \(t\in I_a\),
\(|t-c|\le |t-a|+|a-c|\le13A<r\), so the disk of radius \(r\)
centered at \(c\) contains \(I_a\). We will apply Jensen's formula
on the concentric disk of radius \(2r\). This requires an upper
bound for the function on the larger disk and control of its value
at the center.

Since \(Y_N^x\) is a trigonometric polynomial, it extends to an
entire function, with Taylor expansion
\[
Y_N^x(z)
=\sum_{k=0}^\infty
\frac{R^{(k)}(c)+\boldsymbol v^{(k)}(c)\cdot x}{k!}(z-c)^k.
\]
Put \(M_x=\sup_{|z-c|\le2r}|Y_N^x(z)|\).
By \eqref{eq:conditional-all-derivatives},
\(\E_\eta|R^{(k)}(c)|\le1\) and
\(|\boldsymbol v^{(k)}(c)\cdot x|\le|x|\) for every \(k\ge0\).
Taking absolute values in the series and applying Tonelli's theorem,
we obtain
\begin{align}
\E_\eta M_x
&\le\sum_{k=0}^\infty\frac{(2r)^k}{k!}
\left(\E_\eta|R^{(k)}(c)|
      +|\boldsymbol v^{(k)}(c)\cdot x|\right)
\le e^{2r}(1+|x|).
\label{eq:analytic-supremum-bound}
\end{align}
The factor \(e^{2r}\) is a numerical constant, since \(r=16A\) is fixed.

Next, \(Y_N^x(c)=R(c)+\boldsymbol v(c)\cdot x\) is a real Gaussian
random variable with variance \(\Var(R(c))\ge c_0\). Its density
is therefore bounded by \((2\pi c_0)^{-1/2}\), independently of
its mean. In particular,
\(\mathbb{P}_\eta\{|Y_N^x(c)|<\varepsilon\}\le C\varepsilon\) for every
\(\varepsilon>0\). Writing
\(\log^-u=\max\{-\log u,0\}\), the tail integral formula gives
\begin{equation}
\label{eq:negative-log-at-center}
\E_\eta\log^-|Y_N^x(c)|
=\int_0^\infty\mathbb{P}_\eta\{|Y_N^x(c)|<e^{-u}\}\,du
\le\int_0^\infty\min\{1,Ce^{-u}\}\,du\le C.
\end{equation}
Also, \(Y_N^x(c)\ne0\) almost surely, so the conditional polynomial
is almost surely not identically zero.

Let \(n_x(r)\) be the number of complex zeros in the disk
\(\{|z-c|\le r\}\), counted with multiplicity. Jensen's formula
on the concentric disk of radius \(2r\) shows that each zero in the
inner disk contributes at least \(\log 2\). Hence
\begin{equation}
\label{eq:Jensen-count-inequality}
n_x(r)\log2
\le \log M_x-\log|Y_N^x(c)|
\le \log^+M_x+\log^-|Y_N^x(c)|.
\end{equation}
Here \(\log^+u:=\max\{\log u,0\}\). If the outer circle contains
a zero, apply Jensen's formula at radii increasing to \(2r\) that
avoid zeros and pass to the limit.

Since \(I_a\) lies in the inner disk and \(n_x(r)\) counts zeros with
multiplicity, \(\mathcal Z_N^x(I_a)\le n_x(r)\), including when the
conditioning forces multiple zeros. Moreover,
\(\log^+u\le\log(1+u)\), and concavity of the logarithm gives
\(\E_\eta\log^+M_x\le\log(1+\E_\eta M_x)\).
Taking expectations in \eqref{eq:Jensen-count-inequality} and using
\eqref{eq:analytic-supremum-bound} and
\eqref{eq:negative-log-at-center}, we obtain
\[
\begin{aligned}
\E_\eta\mathcal Z_N^x(I_a)
&\le \frac{\log(1+\E_\eta M_x)+C}{\log2}\le C\{1+\log(1+|x|)\}.
\end{aligned}
\]
The same argument applies to \(I_b:=[b-A,b+A]\cap[0,T_N]\).
Since \(U_{a,b}=I_a\cup I_b\), its zero count is at most the sum
of the two counts, proving \eqref{eq:local-Jensen-bound}.
\end{proof}

\begin{proof}[Proof of Proposition~\ref{prop:replacement}]
Write \(F=[0,T_N]\setminus U_{a,b}\), and let \(I_1,\ldots,I_m\), \(m\le3\), be the closures of its nonempty interval components. Each \(I_j\) satisfies the hypotheses of Lemma \ref{lem:far-count-change}. Every point added by taking these closures has distance at least \(A\) from both conditioning times, so neither conditional field vanishes there almost surely. Consequently, the zero counts on \(F\) agree with the sums of those on the \(I_j\). Since \(U_{a,b}\) and \(F\) are disjoint, their zero counts are additive, including when the conditioning forces a zero at an endpoint of \([0,T_N]\).

Applying Lemma~\ref{lem:far-count-change} to each component of
\([0,T_N]\setminus U_{a,b}\), and using that there are at most three such
components, gives
\[
\begin{aligned}
&\E_\eta\bigl|
\mathcal Z_N^x([0,T_N]\setminus U_{a,b})
-\mathcal Z_N^y([0,T_N]\setminus U_{a,b})
\bigr|\le6+C|x-y|(1+|x|+|y|)
\int_0^{T_N}d_{a,b}(t)\,dt.
\end{aligned}
\]
Moreover,
\begin{equation*}
\int_0^{T_N}d_{a,b}(t)\,dt
\le C L_{T_N},
\end{equation*}
and therefore the far-region contribution is bounded by $6+C|x-y|(1+|x|+|y|)L_{T_N}.$

On \(U_{a,b}\), we simply use $\bigl|\mathcal Z_N^x(U_{a,b})-\mathcal Z_N^y(U_{a,b})\bigr|
\le \mathcal Z_N^x(U_{a,b})+\mathcal Z_N^y(U_{a,b}).$ Lemma~\ref{lem:near-count-bound} then gives
\[
\E_\eta\bigl|
\mathcal Z_N^x(U_{a,b})-\mathcal Z_N^y(U_{a,b})
\bigr|
\le
C\bigl\{
1+\log(1+|x|)+\log(1+|y|)
\bigr\}.
\]
Combining the near- and far-region estimates, and using
\[
|x-y|\le |x|+|y|,
\qquad
1+\log(1+|x|)+\log(1+|y|)
\le C(1+|x|+|y|)^2,
\]
together with \(L_{T_N}\ge1\), yields
\[
\E_\eta|\mathcal Z_N^x-\mathcal Z_N^y|
\le
C(1+|x|+|y|)^2L_{T_N}.
\]
This is \eqref{eq:replacement}.
\end{proof}

\subsection{Covariance estimates after Ornstein--Uhlenbeck averaging}

Proposition~\ref{prop:replacement} controls the exact count of the
conditional fields. The covariance identity
\eqref{eq:local-resampling-identity} concerns \(S_{N,\rho}\), so we
next lift the count estimate through the semigroup. We use the same added Gaussian noise for both conditional
fields. Proposition~\ref{prop:replacement} then yields a
bound independent of \(\rho\), as shown in
Lemma~\ref{lem:OU-stability} below.

We also need a version of \eqref{eq:semigroup-on-count} after fixing
some of the Gaussian coordinates. If \(\boldsymbol{\xi}_N'\) is an independent
copy of \(\boldsymbol{\xi}_N\), Mehler's formula gives
\begin{equation}
 \label{eq:Mehler-coefficients}
 S_{N,\rho}(z)
 =\E_{\boldsymbol{\xi}_N'}S_N(\rho z+\sqrt{s_\rho}\boldsymbol{\xi}_N'),
 \qquad z\in\mathfrak H_N.
\end{equation}
This identity holds pointwise for every deterministic \(z\), with the
left-hand side defined by the integral
\eqref{eq:S-N-rho-definition}. To verify the identity, write the
deterministic polynomial with coefficient vector \(z\) and its
standardized derivative as
\[
 Y_N^z(t)=\langle e_{N,t},z\rangle,
 \qquad
 V_N^z(t)=\langle f_{N,t},z\rangle.
\]
Conditional on \(z\), the added Gaussian noise has independent value
and standardized derivative at every fixed time, both with variance
\(s_\rho>0\). The conditional Rice formula therefore gives
\begin{align*}
 &\E_{\boldsymbol{\xi}_N'}\mathcal Z_N(\rho z+\sqrt{s_\rho}\boldsymbol{\xi}_N')\\*
 &\qquad=\sqrt{\lambda_N}\int_0^{T_N}
 \frac{\exp\{-\rho^2Y_N^z(t)^2/(2s_\rho)\}}
 {\sqrt{2\pi s_\rho}}
 \E\bigl|\rho V_N^z(t)+\sqrt{s_\rho}Z\bigr|\,dt\\*
 &\qquad=\sqrt{\lambda_N}\int_0^{T_N}
 B_\rho(Y_N^z(t))A_\rho(V_N^z(t))\,dt.
\end{align*}
Subtracting the mean in \eqref{eq:mean-zero-count} and dividing by \(\sqrt{T_N}\) proves \eqref{eq:Mehler-coefficients}. All counts in
this calculation are bounded by \(2N\) almost surely, and the
conditional Gaussian field has strictly positive value variance.
Thus the Rice formula and the required integrability apply for
each deterministic \(z\).

\begin{lemma}
\label{lem:OU-stability}
For every \(N\ge N_0\), under the decomposition \eqref{eq:coefficient-decomposition}, the following bound holds uniformly in \(a,b\in[0,T_N]\), \(x,y\in\mathbb R^p\), and \(\rho\in[1/2,1)\),
\begin{equation*}
 \E_\eta\left|S_{N,\rho}(\mathsf Qx+\eta)
              -S_{N,\rho}(\mathsf Qy+\eta)\right|
 \le C\frac{L_{T_N}}{\sqrt{T_N}}(1+|x|+|y|)^2.
\end{equation*}
\end{lemma}

\begin{proof}
In \eqref{eq:Mehler-coefficients}, use the same OU noise \(\mathsf Qz+\eta'\) for the two arguments, where \(z\sim\mathcal
N(0,I_p)\), \(\eta'\) is an independent copy of \(\eta\), and all
these variables are independent. Set \(\widetilde\eta=\rho\eta+\sqrt{s_\rho}\eta'\).
Then \(\widetilde\eta\) has the same law as \(\eta\)
and is independent of \(z\). The new
coordinates in the subspace are \(x_z:=\rho x+\sqrt{s_\rho}z\) and \(y_z:=\rho y+\sqrt{s_\rho}z\).
Conditional Jensen's inequality, followed by
Proposition~\ref{prop:replacement}, gives
\[
 \E_\eta\left|S_{N,\rho}(\mathsf Qx+\eta)
 -S_{N,\rho}(\mathsf Qy+\eta)\right|
 \le C\frac{L_{T_N}}{\sqrt{T_N}}
 \E_z(1+|x_z|+|y_z|)^2.
\]
The centering constants cancel. Since \(p\le5\), the last expectation
is at most \(C(1+|x|+|y|)^2\), uniformly in \(\rho\).
\end{proof}

\begin{lemma}
\label{lem:weighted-test-covariance}
For every \(N\ge N_0\), \(a,b\in[0,T_N]\) and \(\rho\in[1/2,1)\), let \(X\) be
orthonormal Gaussian coordinates for \(\mathfrak H_{a,b}^{(N)}\). If \(k:\mathbb R^p\to\mathbb R\)
is measurable and
\(\E[|k(X)|(1+|X|)^2]<\infty\), then for every Lipschitz function
\(g:\mathbb R\to\mathbb R\),
\begin{equation*}
 \begin{aligned}
 |\Cov(k(X),g(S_{N,\rho}))|
 &\le C\operatorname{Lip}(g)\frac{L_{T_N}}{\sqrt{T_N}}\times\E[|k(X)|(1+|X|)^2].
 \end{aligned}
\end{equation*}
\end{lemma}

\begin{proof}
Let \(X'\) be an independent copy of \(X\), also independent of
\(\eta\). As in \eqref{eq:local-resampling-identity},
\[
\begin{aligned}
\Cov(k(X),g(S_{N,\rho}))
&=\E\Bigl[k(X)\bigl\{
g(S_{N,\rho}(\mathsf QX+\eta))
-g(S_{N,\rho}(\mathsf QX'+\eta))
\bigr\}\Bigr].
\end{aligned}
\]
For fixed \(N\), \(S_{N,\rho}\) is bounded, so the covariance is
well defined under the stated moment assumption. Conditioning on
\(X,X'\), using the Lipschitz property of \(g\), and applying
Lemma~\ref{lem:OU-stability}, we obtain
\[
\begin{aligned}
|\Cov(k(X),g(S_{N,\rho}))|
&\le
C\operatorname{Lip}(g)\frac{L_{T_N}}{\sqrt{T_N}}
\E\bigl[
|k(X)|(1+|X|+|X'|)^2
\bigr].
\end{aligned}
\]
Since \(p\le5\) and \(X'\sim\mathcal N(0,I_p)\),$\E_{X'}(1+|X|+|X'|)^2
\le C(1+|X|)^2.$ Therefore
\[
|\Cov(k(X),g(S_{N,\rho}))|
\le
C\operatorname{Lip}(g)\frac{L_{T_N}}{\sqrt{T_N}}
\E\bigl[|k(X)|(1+|X|)^2\bigr].
\]
\end{proof}

\section{Completion of the Malliavin--Stein estimate}
\label{sec:malliavin-stein}

The replacement argument has given
Lemma~\ref{lem:weighted-test-covariance} for every local function of
the at most five Gaussian coordinates. We apply it to \(K_{N,\rho}(a,b)\).
This requires a weighted first absolute moment, uniform in
\(\rho\), and a separate bound for the endpoint term in
\eqref{eq:gamma-IBP-representation}. We then integrate the resulting
covariance estimate, prove Proposition~\ref{prop:regularized-clt},
and remove the regularization to obtain Theorem~\ref{thm:main}.

\subsection{Uniform weighted norms and test covariances}

The function \(K_{N,\rho}(a,b)\) is measurable with respect to the
Gaussian coordinates generated by \(Y_N(a),Y_N'(a),Y_N''(a),Y_N(b),Y_N'(b)\).
Thus Lemma~\ref{lem:weighted-test-covariance} is applicable once a
weighted first absolute moment has been estimated.

\begin{lemma}
\label{lem:weighted-kernel-bound}
For every sufficiently large \(N\), every \(a,b\in[0,T_N]\), and \(\rho\in[1/2,1)\), let \(X\) be any orthonormal Gaussian coordinate
vector for \(\mathfrak H_{a,b}^{(N)}\). Then
\begin{equation}
 \label{eq:K-weighted-first-moment}
 \E[|K_{N,\rho}(a,b)|(1+|X|)^2]\le C\chi(b-a).
\end{equation}
Moreover,
\begin{equation}
 \label{eq:boundary-L1}
 \E|\mathcal B_{N,\rho}|\le C\frac{L_{T_N}}{T_N}.
\end{equation}
\end{lemma}

\begin{proof}
Choose the first two coordinates of \(X\) to be \(X_1=Y_N(a)\) and \(X_2=V_N(a)\). This is possible because these
two variables are orthonormal in the first Wiener chaos. All other
linear Gaussian variables appearing in \(K_{N,\rho}(a,b)\) are linear
forms in \(X\), with coefficient norms equal to their uniformly
bounded standard deviations. By \eqref{eq:J-kernel-envelope},
\begin{equation*}
 |K_{N,\rho}(a,b)|
 \le C\chi(b-a)B_\rho(X_1)B_\rho(X_2)(1+|X|)^2.
\end{equation*}
The identity
\begin{equation}
 \label{eq:bump-change-of-measure}
 B_\rho(x)\varphi(x)=\varphi(0)\varphi_{s_\rho}(x)
\end{equation}
shows that, in the Gaussian integral, the weight
\(B_\rho(X_1)B_\rho(X_2)\) replaces the first two standard
Gaussian densities by densities of centered Gaussians with
variance \(s_\rho\), up to the fixed factor \(\varphi(0)^2\).
Consequently,
\[
 \E[B_\rho(X_1)B_\rho(X_2)(1+|X|)^4]\le C,
\]
uniformly in \(\rho\), since \(s_\rho\le1\) and \(p\le5\). This proves
\eqref{eq:K-weighted-first-moment}. The weight depends only on \(|X|\),
so the assertion holds in any orthonormal coordinate system.

For the boundary term, use \(|A_\rho'|\le1\) and Gaussian regression:
\[
 \E[1+|V_N(b)|\mid Y_N(a)=x]\le C(1+|x|).
\]
Together with \eqref{eq:J-kernel-envelope} and
\eqref{eq:bump-change-of-measure}, this gives
\[
 \E|B_\rho(Y_N(a))A_\rho'(V_N(a))J_{N,\rho}(a,b)|
 \le C\chi(b-a).
\]
Since
\begin{equation}
 \label{eq:chi-integral}
 \sup_{a\in[0,T_N]}\int_0^{T_N}\chi(b-a)\,db\le CL_{T_N},
\end{equation}
integration in \(b\) and separate treatment of the two endpoints in
\eqref{eq:boundary-definition} prove \eqref{eq:boundary-L1}.
\end{proof}

Combining Lemmas~\ref{lem:weighted-test-covariance} and
\ref{lem:weighted-kernel-bound}, we obtain, for every Lipschitz function
\(g\) and uniformly in \(\rho\in[1/2,1)\),
\begin{equation}
 \label{eq:local-test-covariance}
 |\Cov(K_{N,\rho}(a,b),g(S_{N,\rho}))|
 \le C\operatorname{Lip}(g)\chi(b-a)
 \frac{L_{T_N}}{\sqrt{T_N}}.
\end{equation}

\subsection{Normal approximation}

\begin{proof}[Proof of Proposition~\ref{prop:regularized-clt}]
Let \(h\) be a smooth function with \(\operatorname{Lip}(h)\le1\),
and let \(f_h\) solve its normal Stein equation. Set
\(g(u)=f_h'(u/\sigma_{N,\rho})\). By
\eqref{eq:direct-Stein-covariance},
\[
 |\E h(S_{N,\rho}/\sigma_{N,\rho})-\E h(Z)|
 =\sigma_{N,\rho}^{-2}|\Cov(\Gamma_{N,\rho},g(S_{N,\rho}))|.
\]
By \eqref{eq:Stein-solution-bounds},
\(\|f_h'\|_\infty\le1\) and \(\operatorname{Lip}(g)\le2/\sigma_{N,\rho}\).
From
\eqref{eq:gamma-IBP-representation}, \eqref{eq:boundary-L1}, and
\eqref{eq:local-test-covariance},
\begin{align*}
 |\Cov(\Gamma_{N,\rho},g(S_{N,\rho}))|
 &\le2\|f_h'\|_\infty\E|\mathcal B_{N,\rho}|+\frac2{T_N}\int_{[0,T_N]^2}
 |\Cov(K_{N,\rho}(a,b),g(S_{N,\rho}))|\,da\,db\\*
 &\le C\frac{L_{T_N}}{T_N}
 +\frac{CL_{T_N}}{\sigma_{N,\rho} T_N^{3/2}}
 \int_{[0,T_N]^2}\chi(b-a)\,da\,db\\*
 &\le C\frac{L_{T_N}}{T_N}+C\frac{L_{T_N}^2}{\sigma_{N,\rho}\sqrt{T_N}}.
\end{align*}
The last step uses \eqref{eq:chi-integral}. All interchanges are
justified for fixed \(N\) and \(\rho<1\) by the polynomial moment
bounds. The displayed estimates also give absolute integrability of
the covariance terms being integrated.

Taking the supremum over the test functions, and using smooth
approximation to include every \(1\)-Lipschitz function, gives the
more explicit bound
\begin{equation*}
 d_{\mathrm W}(S_{N,\rho}/\sigma_{N,\rho},Z)
 \le C\left\{\frac{L_{T_N}}{\sigma_{N,\rho}^2T_N}
 +\frac{L_{T_N}^2}{\sigma_{N,\rho}^3\sqrt{T_N}}\right\}.
\end{equation*}
The third power of \(\sigma_{N,\rho}\) in the second term includes the
Lipschitz constant of \(u\mapsto f_h'(u/\sigma_{N,\rho})\).
If \(\sigma_{N,\rho}^2\ge c_0\), the right-hand side is at most
\(C(c_0)L_{T_N}^2/\sqrt{T_N}\). This proves
\eqref{eq:regularized-CLT-bound}.
\end{proof}

The argument is an application of the Malliavin--Stein identity, but it
does not require a bound for \(\Var(\Gamma_{N,\rho})\) that is uniform
as \(\rho\uparrow1\). The quantity estimated is its covariance with
the Stein test function. The two logarithmic factors arise from
\eqref{eq:replacement} and the final integration in
\eqref{eq:chi-integral}, respectively.

\subsection{Proof of the main theorem}
\label{subsec:proof-main}

We now pass from \(S_{N,\rho}\) to \(S_N\).
The uniformity in \(\rho\) established in
Proposition~\ref{prop:regularized-clt} allows us to use only the convergence for fixed \(N\) in \eqref{eq:fixed-N-OU-limit}.

\begin{proof}[Proof of Theorem~\ref{thm:main}]
By \eqref{eq:known-variance-asymptotics}, there exists \(v_*>0\) such
that \(v_N\ge v_*\) for all sufficiently large \(N\). Fix such an \(N\). Lemma~\ref{lem:regularized-chaos-representation} and the reverse
triangle inequality in \(L^2\) give
\begin{equation*}
 |\sigma_{N,\rho}-v_N|
 \le\|S_{N,\rho}-S_N\|_2\longrightarrow0
 \qquad\text{as }\rho\uparrow1.
\end{equation*}
Thus \(\sigma_{N,\rho}\ge v_*/2\) for \(\rho\) sufficiently close
to one. The point beyond which this holds may depend on \(N\), but
the lower bound \(v_*/2\) does not.

Under the natural coupling,
\[
 \left\|\frac{S_N}{v_N}-
 \frac{S_{N,\rho}}{\sigma_{N,\rho}}\right\|_1
 \le\frac{\|S_N-S_{N,\rho}\|_2}{v_N}
 +\frac{|\sigma_{N,\rho}-v_N|}{v_N}\longrightarrow0.
\]
Using Proposition~\ref{prop:regularized-clt} with
\(c_0=v_*^2/4\) and then letting \(\rho\uparrow1\), we obtain
\[
 d_{\mathrm W}(S_N/v_N,Z)
 \le C\frac{(\log T_N)^2}{\sqrt{T_N}}.
\]
The constant is independent of \(N\), since the estimate for the
regularized statistic was uniform on
\(\{\sigma_{N,\rho}\ge v_*/2\}\). No interchange of the limits in \(N\) and \(\rho\) is involved: the regularization is removed for
each fixed \(N\). Finally, \(T_N=N\pi\) and
\[
 \frac{S_N}{v_N}=
 \frac{\mathcal Z_N-\E\mathcal Z_N}
 {\sqrt{\Var(\mathcal Z_N)}}.
\]
This proves the upper bound in \eqref{eq:main-bound}.

We finish with the lower bound, which uses only the lattice support and
the order of the variance. Let \(h_N=(\sqrt{T_N}\,v_N)^{-1}\) be the lattice spacing of the
standardized zero count. By \eqref{eq:known-variance-asymptotics},
\(h_N\asymp N^{-1/2}\). Since \(\mathcal Z_N\) is integer-valued,
\(S_N/v_N\) is supported on
\[
 \Lambda_N:=\left\{h_N(k-\E\mathcal Z_N):
 k\in\mathbb Z\right\}.
\]
The function \(g_N(x):=\operatorname{dist}(x,\Lambda_N)\) is
\(1\)-Lipschitz and vanishes on this support. Hence
\[
 d_{\mathrm W}(S_N/v_N,Z)\ge\E g_N(Z).
\]
For all sufficiently large \(N\), \(h_N\le1/2\), and there are at
least \(h_N^{-1}\) complete lattice cells contained in \([-1,1]\).
On every cell \([u,u+h_N]\),
\[
 \int_u^{u+h_N}\operatorname{dist}(x,\Lambda_N)\,dx
 =\int_0^{h_N}\min\{t,h_N-t\}\,dt=\frac{h_N^2}{4}.
\]
The standard Gaussian density is bounded below by \(\varphi(1)>0\)
on \([-1,1]\). Summing over these cells gives
\[
 \E g_N(Z)\ge\frac{\varphi(1)}4h_N\ge cN^{-1/2}.
\]
Together with the upper bound, this proves
\eqref{eq:main-bound}.
\end{proof}

\appendix

\section{Uniform conditioning and projection estimates}
\label{app:projection-estimates}

This appendix proves Lemma~\ref{lem:projection-tail}. The main point is
to choose coordinates for \(\mathfrak H_{a,b}^{(N)}\) whose covariance
matrix remains uniformly nondegenerate as \(b\to a\). All covariance
decay estimates below involve time differences in \([-T_N,T_N]\).

\subsection{Uniformly conditioned coordinates}

Throughout this appendix, write \(h=b-a\). By stationarity, the
relevant covariance matrices depend on \(a,b\) only through \(h\).

We establish the following uniform conditioning statement.
There exist constants \(h_0\in(0,1/4]\), \(0<c\le C<\infty\), and
\(N_0\ge1\) such that, for every \(N\ge N_0\) and
\(a,b\in[0,T_N]\), the coordinate vector specified below spans
\(\mathfrak H_{a,b}^{(N)}\) and has covariance eigenvalues in \([c,C]\).
The chosen vector is \((Y_N(a),Y_N'(a),Y_N''(a))\) when \(h=0\),
the adapted vector \eqref{eq:adapted-five-coordinates} when
\(0<|h|\le h_0\), and the raw vector
\eqref{eq:raw-five-coordinates} when \(|h|>h_0\).

\begin{proof}
When \(h=0\), uniform conditioning follows from
\eqref{eq:three-jet-covariance}. For \(h\ne0\), define
\begin{align}
D_{3,N}(a,b)
&=\frac{Y_N(b)-Y_N(a)-hY_N'(a)-\frac12h^2Y_N''(a)}{h^3},
\notag\\*
D_{4,N}(a,b)
&=\frac{Y_N'(b)-Y_N'(a)-hY_N''(a)-3h^2D_{3,N}(a,b)}{h^3}.
\label{eq:D4-definition}
\end{align}
Taylor's formula with integral remainder gives
\begin{align}
D_{3,N}(a,b)
&=\frac12\int_0^1(1-u)^2Y_N^{(3)}(a+uh)\,du,
\label{eq:D3-integral}\\*
D_{4,N}(a,b)
&=\frac12\int_0^1u(1-u)^2Y_N^{(4)}(a+uh)\,du.
\label{eq:D4-integral}
\end{align}
For the second identity, the numerator in
\eqref{eq:D4-definition} equals
\[
h^2\int_0^1
\left\{(1-u)-\frac32(1-u)^2\right\}
Y_N^{(3)}(a+uh)\,du.
\]
The polynomial in braces has primitive \(-u(1-u)^2/2\), which
vanishes at both endpoints, so integration by parts yields
\eqref{eq:D4-integral}, also for \(h<0\).

Conversely,
\[
\begin{aligned}
Y_N(b)
&=Y_N(a)+hY_N'(a)+\tfrac12h^2Y_N''(a)+h^3D_{3,N}(a,b),\\
Y_N'(b)
&=Y_N'(a)+hY_N''(a)+3h^2D_{3,N}(a,b)+h^3D_{4,N}(a,b).
\end{aligned}
\]
Hence
\begin{equation}
\label{eq:adapted-five-coordinates}
\bigl(Y_N(a),Y_N'(a),Y_N''(a),
D_{3,N}(a,b),D_{4,N}(a,b)\bigr)
\end{equation}
spans the same Gaussian space as the raw two-time coordinates.

Introduce the symmetric spectral measure
\(\mu_N=(2N)^{-1}\sum_{n=1}^N
(\delta_{n/N}+\delta_{-n/N})\).
After removing the common phase at \(a\), the spectral functions of
\eqref{eq:adapted-five-coordinates} converge uniformly on \([-1,1]\),
as \(h\to0\), to
\begin{equation}
\label{eq:limiting-five-jet}
1,\qquad i\xi,\qquad(i\xi)^2,\qquad
\frac{(i\xi)^3}{6},\qquad\frac{(i\xi)^4}{24}.
\end{equation}
Indeed, \eqref{eq:D3-integral}--\eqref{eq:D4-integral} and
\(|e^{iuh\xi}-1|\le |h|\) for \(|\xi|\le1\) give an \(O(|h|)\)
error uniformly in \(N\).

Let \(\Sigma_N(0)\) be the Gram matrix of the five functions in
\eqref{eq:limiting-five-jet} under \(\mu_N\). After reordering the
even and odd coordinates, it is the direct sum of
\[
E_N:=
\begin{pmatrix}
1&-\alpha_N&\beta_N/24\\
-\alpha_N&\beta_N&-\gamma_N/24\\
\beta_N/24&-\gamma_N/24&\delta_N/576
\end{pmatrix},
\qquad
O_N:=
\begin{pmatrix}
\alpha_N&-\beta_N/6\\
-\beta_N/6&\gamma_N/36
\end{pmatrix}.
\]
The leading principal minors of \(E_N\) converge to
\(1\), \(4/45\), and \(4/4465125\), while
\(\alpha_N\to1/3\) and
\(\det O_N=(\alpha_N\gamma_N-\beta_N^2)/36\to1/4725\).
Thus the limiting matrices of both blocks are positive definite.
Consequently, for all sufficiently large \(N\), the eigenvalues of
\(\Sigma_N(0)\) lie in a fixed interval
\([c,C]\subset(0,\infty)\).

If \(\Sigma_N(h)\) denotes the covariance matrix of
\eqref{eq:adapted-five-coordinates}, then
\eqref{eq:D3-integral}--\eqref{eq:D4-integral} imply
\(\|\Sigma_N(h)-\Sigma_N(0)\|_{\mathrm{op}}\le C|h|\).
Choosing a sufficiently small fixed \(h_0\in(0,1/4]\) therefore
gives uniform conditioning whenever \(0<|h|\le h_0\).

For \(h_0\le |h|\le T_N\), use instead the raw coordinates
\begin{equation}
\label{eq:raw-five-coordinates}
\bigl(Y_N(a),Y_N'(a),Y_N''(a),Y_N(b),Y_N'(b)\bigr).
\end{equation}
After removing the common phase at \(a\), their spectral functions are
\(1,i\xi,(i\xi)^2,e^{ih\xi},i\xi e^{ih\xi}\). For every \(h\ne0\)
these functions are linearly independent. Indeed, a relation has the
form \(P(\xi)+e^{ih\xi}Q(\xi)=0\), with \(\deg P\le2\) and
\(\deg Q\le1\). Applying \((\partial_\xi-ih)^2\) gives
\((\partial_\xi-ih)^2P=0\). Since \(h\ne0\), this operator is
invertible on polynomials of degree at most two, so \(P=Q=0\).

Let \(\mu(d\xi)=\frac12\mathbf1_{[-1,1]}(\xi)\,d\xi\).
For every fixed \(H<\infty\), the Gram matrices under \(\mu_N\)
converge uniformly, for \(h_0\le |h|\le H\), to the corresponding
Gram matrices under \(\mu\). The latter are positive definite for
every \(h\ne0\), and their least eigenvalues depend continuously on
\(h\). Compactness therefore gives uniform conditioning on
\(h_0\le |h|\le H\), for all sufficiently large \(N\).

For \(H\le |h|\le T_N\), the covariance matrix of
\eqref{eq:raw-five-coordinates} is a perturbation of the direct sum
of the covariance matrices of \((Y_N,Y_N',Y_N'')\) and
\((Y_N,Y_N')\). The diagonal blocks are uniformly positive definite,
while Lemma~\ref{lem:uniform-covariance-decay} bounds the
off-diagonal block by \(C/|h|\). Choosing \(H\) sufficiently large,
and then increasing \(N_0\) if necessary, proves uniform conditioning
throughout this remaining regime.
\end{proof}

\subsection{The projection bound}

\begin{proof}[Proof of Lemma~\ref{lem:projection-tail}]
For each pair \((a,b)\), let \(\boldsymbol{\mathcal X}_{a,b}\) be the three raw
coordinates when \(a=b\), the adapted vector
\eqref{eq:adapted-five-coordinates} when \(0<|b-a|\le h_0\), and the
raw vector \eqref{eq:raw-five-coordinates} otherwise. Write
\(\Sigma_{a,b}=\Cov(\boldsymbol{\mathcal X}_{a,b})\) and
\(\boldsymbol c_j(t)=\Cov(\boldsymbol{\mathcal X}_{a,b},Y_N^{(j)}(t))\).
Since \(\boldsymbol{\mathcal X}_{a,b}\) spans \(\mathfrak H_{a,b}^{(N)}\), the
Gram-matrix formula for orthogonal projection, together with the
uniform conditioning established in the preceding subsection, gives
\begin{equation}
\label{eq:projection-Gram-formula}
\bigl\|\Pi_{a,b}^{(N)}e_{N,t}^{(j)}\bigr\|_{\mathfrak H}^2
=
\boldsymbol c_j(t)^\top\Sigma_{a,b}^{-1}\boldsymbol c_j(t)
\le C|\boldsymbol c_j(t)|^2.
\end{equation}

Every coordinate \(X_i\) of \(\boldsymbol{\mathcal X}_{a,b}\) can be written as
\(X_i=\int Y_N^{(r_i)}(u)\,\nu_i(du)\), where
\(0\le r_i\le4\) and \(\|\nu_i\|_{\mathrm{TV}}\le C\).
Here \(\|\nu\|_{\mathrm{TV}}\) denotes the total variation norm of a
finite signed measure \(\nu\). For raw coordinates, \(\nu_i\) is a point mass at \(a\) or \(b\);
for adapted coordinates,
\eqref{eq:D3-integral}--\eqref{eq:D4-integral} show that \(\nu_i\)
is supported on the segment joining \(a\) and \(b\), whose length is
at most \(h_0\).

Hence, for \(j=0,1,2\), every component of
\(\boldsymbol c_j(t)\) is, up to a sign, an integral of
\(r_N^{(r_i+j)}(t-u)\) against such a measure. In the adapted case,
\(\chi(t-u)\le C\{\chi(t-a)+\chi(t-b)\}\) uniformly for \(u\) on
this segment; for raw coordinates the same estimate is immediate.
Since \(r_i+j\le6\), Lemma~\ref{lem:uniform-covariance-decay} yields
\[
|\boldsymbol c_j(t)|
\le C\{\chi(t-a)+\chi(t-b)\},
\qquad j=0,1,2.
\]
Substituting this bound into \eqref{eq:projection-Gram-formula}
proves \eqref{eq:projection-tail}.
\end{proof}

\section{Details for the regularized kernels}
\label{app:soft-kernel-details}

This appendix verifies the regularity used in
Subsection~\ref{subsec:kernel-representation}
and proves Lemma~\ref{lem:kernel-estimates}.

\begin{proof}[Proof of Lemma~\ref{lem:kernel-estimates}]
For fixed \(\rho\in[1/2,1)\), the functions
\(B_\rho\) and \(B_\rho'\) are bounded, while
\(A_\rho(y)\le C(1+|y|)\) and \(|A_\rho'(y)|\le1\).
Since
\[
\partial_1F_\rho=B_\rho'A_\rho,
\qquad
\partial_2F_\rho=B_\rho A_\rho',
\]
the function \(F_\rho\) is continuously differentiable
and satisfies
\[
|F_\rho(x,y)|+|\nabla F_\rho(x,y)|
\le C_\rho(1+|y|).
\]

For \(0<u<1\), the explicit
formulas give
\[
 B_u'(x)=-\frac{u^2x}{s_u}B_u(x),\qquad
 |A_u'(y)|\le u,\qquad A_u(y)\le C(1+|y|).
\]
For fixed \(\rho<1\), differentiation under \(G_\rho=\int_0^\rho u^{-1}F_u\,du\) is legitimate: on compact
sets of \((x,y)\), the differentiated integrands are integrable at
\(u=0\), as is seen from these formulas. We obtain
\begin{align*}
 |\partial_xG_\rho(x,y)|
 &\le C(1+|y|)|x|\int_0^\rho
 u(1-u^2)^{-3/2}
 \exp\left(-\frac{u^2x^2}{2(1-u^2)}\right)du,\\*
 |\partial_yG_\rho(x,y)|
 &\le C\int_0^\rho(1-u^2)^{-1/2}
 \exp\left(-\frac{u^2x^2}{2(1-u^2)}\right)du.
\end{align*}
Use \(q=u/\sqrt{1-u^2}\), so that \(du=(1+q^2)^{-3/2}\,dq\). The first integral, including its
factor \(|x|\), is bounded by
\[
 |x|\int_0^\infty\frac q{\sqrt{1+q^2}}
 e^{-x^2q^2/2}\,dq\le C.
\]
For \(x\ne0\), this follows from \(q/\sqrt{1+q^2}\le1\) and the substitution \(w=|x|q\).
When \(x=0\), the original \(\partial_xG_\rho\) integrand is
identically zero. The second integral is bounded by
\[
 \int_0^\infty\frac{e^{-x^2q^2/2}}{1+q^2}\,dq\le C.
\]
This proves \(|\nabla G_\rho(x,y)|\le C(1+|y|)\), uniformly in \(\rho\). Together with the already established \(L^2(\gamma_2)\)
representation of \(G_\rho\), it also gives its Gaussian Sobolev
regularity.
\end{proof}

\paragraph{Acknowledgements.} 
This work was partially supported by the National Natural Science
Foundation of China (grants 12271475 and U23A2064).

\paragraph{Statement on artificial-intelligence assistance.}
GPT-6 Pro was used for language polishing, LaTeX assistance, and
occasional checking of technical details. It also provided a
lattice-based derivation of the \(N^{-1/2}\) lower bound.
All AI-assisted material was independently verified by the authors,
who take full responsibility for the manuscript.

\end{document}